\documentclass[10pt]{article}
\usepackage{hyperref}
\usepackage{amsfonts}
\usepackage{graphicx}

\usepackage{amssymb,amsmath,amsthm}
\usepackage{nicefrac}
\usepackage{color,xcolor}
\usepackage{mathtools}
\usepackage{graphicx,float}
\usepackage{subfigure}
\usepackage{epstopdf}
\usepackage{caption}
\definecolor{Nblue}{HTML}{0080ff}
\definecolor{NewGreen}{rgb}{0, 0.501, 0}
\definecolor{Red1}{rgb}{0.858, 0.188, 0.478}
\usepackage{hyperref}
\hypersetup{
     colorlinks=true,
    citecolor=Nblue,
    linkcolor=NewGreen,
    filecolor=magenta,      
    urlcolor=Red1
    }
\usepackage[left=2cm,right=2cm,top=2cm,bottom=3cm 
]{geometry}
\usepackage{cite}
\usepackage{environ}

\NewEnviron{neweq}{%
\begingroup
\allowdisplaybreaks
\begin{align}
\begin{split}
    \BODY
    \end{split}
\end{align}
\endgroup}

\NewEnviron{neweq_no}{%
\begingroup
\allowdisplaybreaks
\begin{align*}
    \BODY
\end{align*}
\endgroup}
\newtheorem{lemma}{Lemma}
\numberwithin{lemma}{section}
\newtheorem{proposition}{Proposition}
\numberwithin{proposition}{section}

\numberwithin{definition}{section}

\numberwithin{example}{section}
\newtheorem{theorem}{Theorem}
\numberwithin{theorem}{section}

\numberwithin{corollary}{section}

\numberwithin{remark}{section}

\usepackage{xcolor}

\def\bW{{\bf W}}
\def\bL{{\bf L}}
\def\bI{{\bf I}}
\def\bM{{\bf M}}

\def\bX{{\bf X}}
\def\bv{{\bf v}}

\def\e{x}
\def\i{y}
\def\w{z}
\def\r{{r}}
\def\dss{\displaystyle}

\def\diag{{\rm diag}}
\def\gavg{\epsilon_{\rm m}}
\def\gavg{\tau_{\rm m}}
\def\var{\xi}
\def\sinhc{{\rm sinhc}}
\def\sinc{{\rm sinc}}
\DeclarePairedDelimiter\floor{\lfloor}{\rfloor}

\title{Distributed Delay and Desynchronization in a Brain Network Model}
\date{ }
\author{
Isam Al-Darabsah\footnote{Department of Mathematics and Statistics, Jordan University of Science and Technology, Irbid 22110, Jordan.}\textsuperscript{ ,}{$^\S$}
\hspace{0.5cm}
Sue Ann Campbell\footnote{Department of Applied Mathematics and Centre for Theoretical Neuroscience, University of Waterloo, Waterloo, Canada.}\textsuperscript{ ,}{$^\S$}
 \hspace{0.5cm}
Bootan Rahman\footnote{
Mathematics Unit, School of Science and Engineering, University of Kurdistan Hewler (UKH), Erbil 44001, Iraq.}
\textsuperscript{,}\footnote{Email addresses: imaldarabsah@just.edu.jo (I. Al-Darabsah), sacampbell@uwaterloo.ca (S.A. Campbell), bootan.rahman\linebreak @ukh.edu.krd (B. Rahman).}
\textsuperscript{,}\footnote{Corresponding author.}
}

\begin{document}

\maketitle
\begin{abstract}
{We consider a neural field model which consists of a network of an arbitrary number of Wilson-Cowan nodes with homeostatic adjustment of the inhibitory coupling strength and time delayed, excitatory coupling.  
We extend previous work on this model to include distributed time delays with commonly used kernel distributions: delta function, uniform distribution and gamma distribution.  Focussing on networks which satisfy a constant row sum condition, we show how each eigenvalue of the connectivity matrix may be related to a Hopf bifurcation and that the eigenvalue determines whether the bifurcation leads to synchronized or desynchronized oscillatory behaviour.  We consider two example networks, one with all real eigenvalues (bi-directional ring) and one with some complex eigenvalues (uni-directional ring). In bi-directional rings, the Hopf curves are organized so that only the synchronized Hopf leads to asymptotically stable behaviour. Thus the behaviour in the network is always synchronous. In the uni-directional ring networks, however, intersection points of asynchronous and synchronous Hopf curves may occur resulting in double Hopf bifurcation points. Thus asymptotically stable synchronous and asynchronous limit cycles can occur as well as torus-like solutions which combine synchronous and asynchronous behaviour. Increasing the size of the network or the mean time delay makes these intersection points, and the associated asynchronous behaviour, more likely to occur. Numerical approaches are used to confirm the findings, with Hopf bifurcation curves plotted using \textit{Wolfram Mathematica}. These insights offer a deeper understanding of the mechanisms underlying desynchronization in large networks of oscillators.
}

\end{abstract}

\section{Introduction}
The human central nervous system is responsible for controlling all movements and biological processes that occur in the human body, and consists of highly specialized cells called nerves or neurons. There are approximately 100 billion neurons and an estimated 100 trillion synapses among them, they come in a variety of shapes, sizes and properties \cite{anderson1995}. All activities, such as movement, perception and conscious experience manifest themselves in rhythmic brain oscillations, and disruption or increased activity of neural networks can lead to various brain pathologies. Neurodegenerative diseases, including Alzheimer's disease, epilepsy, Parkinson's disease, and epileptic seizures, selectively disrupt these networks, affecting various neuronal functions \cite{gotz2009}.

The pioneering work of Wilson and Cowan \cite{wilson1972}, which described the time evolution of the mean level of activity of coupled excitatory and inhibitory populations of neurons, has been successfully used to understand many problems in the dynamics of brain networks, such as visual hallucinations \cite{ermentrout1979,pearson2016sensory}, the existence of beta oscillations in the basal ganglia \cite{kim2014} and epileptic activity \cite{meijer2015modeling}. This model and its various extensions, including modifications incorporating time delays, have played a very important role in the analysis of neural populations \cite{destexhe2009wilson,hlinka2012using,wilson2021evolution,coombes2009delays}.

In neural network modelling, time delays are used to take into account the fact that in the majority of biological neural networks, some processes do not happen instantaneously due to a finite propagation velocity of neural signals, times required for information processing, and so forth. Undoubtedly, their inclusion is a non-negligible component of the process and leads to qualitatively new behaviour in the dynamical system not observed in the same system without the time delays \cite{al2020, al2021impact, rahman2017aging,ryu2020}. 
The presence of time delays can profoundly impact system behavior, leading to instability and bifurcations or stabilization of unstable states through the use of time-delayed feedback \cite{pyragas1995control} or synchronization enhancement or suppression \cite{kyrychko2014synchronization}.
{In many cases, delays in natural systems are not fixed but instead are distributed, with the delay expressed as a convolution integral with a delay kernel \cite{atay2003distributed, wille2014synchronization, kyrychko2013amplitude, atay2006neural, liang2009phase, meyer2008distributed}. This is particularly true for complex networks and biological systems, where delays can arise from a variety of sources, including finite signal transmission and processing times or external feedback loops. Understanding the impact of such distributed delays on system behavior is crucial for developing accurate models of complex natural systems.}

Another fascinating aspect of research in neural models is the introduction of homeostatic plasticity of synaptic weights to regulate the firing rate of neural populations. Typically this is implemented by slow variations in synaptic weights, and attempts to prevent the population firing rate from becoming too high or too low. For example, the model of Vogels et al. \cite{vogels2011inhibitory} regulates the excitatory population firing rate by modifying the weights
from the inhibitory neurons to the excitatory neurons. While the fundamental
idea is to stabilize the population firing rate to an equilibrium point, introducing slowly varying synaptic weights can lead to more complex dynamics \cite{hellyer2016local, nicola2018chaos}.

We consider a network of $N$ Wilson-Cowan nodes incorporating the homeostatic plasticity model of \cite{vogels2011inhibitory,hellyer2016local}, in the form considered in \cite{al2021impact,nicola2021,nicola2018chaos}. 
These latter papers investigated  the effect of the coupling matrix structure and the presence of time delays on synchronization in this brain network model. In particular,  
the authors in \cite{nicola2018chaos} considered networks where the connectivity matrix has constant row sum and showed that certain structured networks (e.g. lattice and uni-directional ring) could exhibit desynchronized solutions if the network was large enough. By using the Master Stability Function formalism \cite{pecora1990synchronization}, they linked this desynchronization to the size of the second largest eigenvalue of the connectivity matrix. Using the same approach applied to the network with time delays in the connections between the nodes, 
the authors in \cite{al2021impact} showed that small time delays could synchronize large networks with symmetric connectivity matrices, but not for networks with nonsymmetric connectivity matrices. In fact, small networks which were synchronized without delay could be desynchronized by it.  The authors  attributed the synchronization to a change of the synchronized solution induced by the delay, however the mechanism for the desynchronization by delay was not explained. 

It is the purpose of this paper to investigate the dynamical mechanism for desynchronization in these networks and elucidate why the delay has a different effect on symmetric vs nonsymmetric connectivity matrices. To further the impact of this study we extend the model to include distributed delays. The plan for the article is as follows. Section 2 introduces the model. Section 3 analyzes the stability of the fundamental
equilibrium point and describes curves in parameter space where Hopf bifurcations may occur. Section 4 gives cases studies of these curves for
two particular examples, uni-directional and bi-directional rings, showing how the size of the network and the mean delay in the connections between the nodes influences the geometry of these curves. Based on this geometry we propose a hypothesis for the differential effect of delays on the synchronization of the network. We then use numerical bifurcation analysis and numerical simulations to support this hypothesis. In section 5 we discuss our work and future directions.

\section{Model description}
Expanding the Wilson-Cowan model \cite{wilson1972} and its modification considered in \cite{al2021impact,hellyer2016local,nicola2018chaos, nicola2021} by introducing a distribution of time delays in the excitatory connections between the nodes, yields

\begin{eqnarray}\label{eq1}
\begin{array}{rcl}
		\displaystyle\tau_1\frac{dE_i}{dt} & = & -E_i+\phi (\sum_{j=1}^{N} W_{ij}^{EE} \int_0^{\infty} E_j(t-s)g(s)\, ds- W_i^{EI}I_i ) \,, \\[2mm]
		\displaystyle\frac{dI_i}{dt} & = &   -I_i+\phi(W^{IE}E_i), \\[2mm]
		\displaystyle\tau_2\frac{dW_i^{EI}}{dt} & = & I_i(E_i-p ). 
		\end{array}
\end{eqnarray}
Here $E_i$ is the activity of the excitatory population of neurons within the $i^{th}$ node, $I_i$ is the activity of the inhibitory population in the $i^{th}$ node, $W_i^{EI}$ is the homeostatically adjusted inhibitory to excitatory weight within the $i^{th}$ node, $W_{IE}>0$ is the fixed excitatory to inhibitory weight within the $i^{th}$ node and  $W_{ij}^{EE}\ge 0$ are the (fixed) excitatory to excitatory weights between the nodes. The function  $\phi$  is the  transfer function which determines the proportion of the population of neurons which is active in node $i$. It is assumed to be sigmoidal, thus is strictly increasing and satisfies $0\le \phi(x)\le 1$.
Following \cite{al2021impact,nicola2021,nicola2018chaos} we use the logistic function	
	\begin{eqnarray}\label{eq3}
	\phi(x)=\frac{1}{1+e^{-ax}}
		\end{eqnarray}	
where $a$ controls the steepness of the sigmoid. We assume there is no excitatory self-coupling in the nodes, $W^{EE}_{ii}=0$, and that 
 the between node connectivity matrix satisfies a constant row sum condition
\begin{equation}
\sum_{j=1}^N W^{EE}_{ij}=W^E.
\label{rowsum}
\end{equation}
This corresponds to assuming that each node receives the same amount of input from the other nodes.
 
 We assume the time delays in the propagation of signals from different nodes are the same and described by the delay distribution kernel $g(.)$, satisfying	
	\begin{eqnarray}\label{eq2}
		g(s)\geq 0\qquad \text{and}\qquad  \int_0^{\infty} g(s) ds=1.
		\end{eqnarray}
We note the following properties of this distribution:
\begin{align}
     &\text{Mean delay:}\quad  \gavg=\int_0^{\infty} s g(s) ds,\label{mean_delay_def}\\
     &\text{Variance:}\quad    \var=\int _{0}^{\infty }(s-\gavg)^2\, g(s)ds.\label{variance_def}
\end{align}
 The delay distribution kernel  $g(s)$ can take different forms, here we consider three types.
\begin{itemize}
\item Dirac Delta distribution.
For this distribution, we distinguish two important cases. The first case is $g(s)=\delta(s)$, which is characterised by the following two properties (see Figure  \ref{fig_kernel_distribution}A)
\[
\delta(s)=\left\{\begin{array}{ccc}
\displaystyle{0}& \mbox{for} & s\neq 0, \\ 
\displaystyle{\infty} &  \mbox{for} & s=0,
\end{array}\right. \qquad \mbox{and} \qquad \int_{-\infty}^{\infty} f(s)\delta(s)ds=0.
\]
In this case the system \eqref{eq1} reduces to the system with no delay, studied in \cite{nicola2018chaos}.

The second case, $g(u)=\delta(u-\gavg)$, the distribution  is shifted by $\gavg$ to the right, and has the following properties  (see Figure  \ref{fig_kernel_distribution}A)
\[
\delta(s- \gavg)=\left\{\begin{array}{ccc}
\displaystyle{0}& \mbox{for} & s\neq  \gavg, \\ 
\displaystyle{\infty} &  \mbox{for} & s= \gavg,
\end{array}\right. \qquad \mbox{and} \qquad \int_{-\infty}^{\infty} f(s)\delta(s- \gavg)ds=f(\tau).
\]
 Inserting this kernel into the equation \eqref{eq1}, the coupling takes the form of a discrete time delay $(E_i(t- \gavg) - W_i^{EI}I_i )$ studied in \cite{al2021impact}.

 \item  Uniform distribution. This kernel can be written in the form  (see Figure  \ref{fig_kernel_distribution}B)
\begin{equation}\label{unieq1}
	g(s)=\left\{\begin{array}{ccc}
		\displaystyle{\frac{1}{2\sigma} }& \mbox{for} &  \gavg-\sigma \leq s \leq \gavg+\sigma, \\ \\
		\displaystyle{0} & & \mbox{otherwise},
	\end{array}\right. 
\end{equation}
which has the mean time delay $ \gavg$ and the variance  $\var= {\sigma^2}/{3}$. The parameter $\sigma$ controls the width and height of the distribution, and must satisfy $0< \sigma\le \tau_m$.

\item Gamma distribution.
We write this kernel as follows
\begin{equation}\label{gammaeq1}
	\displaystyle{g(s)=g_{\gamma}^m(s)=\frac{s^{m-1}\gamma^{m} e^{-\gamma s}}{(m-1)!}}.
\end{equation}
It is a distribution with an integer shape (order) parameter $m\in\mathbb{Z}$, ($m\geq 0$) and scale parameter $\gamma \in\mathbb{R}$, ($\gamma>0$).  The mean time delay is  $\gavg={m}/{\gamma}$  
and the variance is $\var={m}/{\gamma^2}$.
When $m=1$ the kernel is typically called {\it weak} kernel or exponential distribution kernel (see Figure  \ref{fig_kernel_distribution}C) while it is called  {\it strong} kernel when $m=2$ (see Figure  \ref{fig_kernel_distribution}D).
\end{itemize}

\begin{figure}[!htb]
     \centering
         \includegraphics[width=1\textwidth]{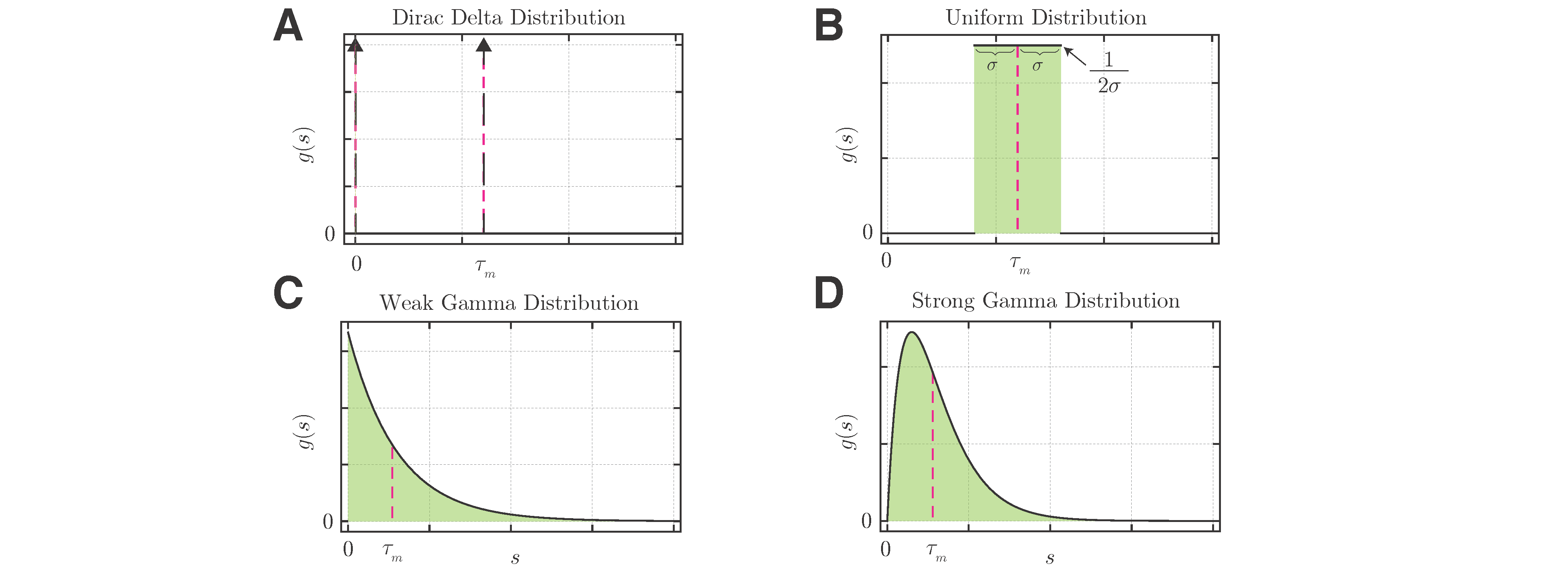}
         \caption{\textbf{Delay distribution kernel $g(s)$ with the  mean time delay $\gavg$}.
         The blue curve represents   $g(s)$, the pink vertical line is   $\gavg$, and the green region shows the area under  $g(s)$.
         \textbf{(A)} Dirac Delta kernel: No delay $g(s)=\delta(s)$ and fixed delay $g(s)=\delta(s-\gavg)$.
        \textbf{(B)} Uniform distribution kernel defined in \eqref{unieq1}.
        \textbf{(C)} Weak gamma distribution kernel defined in \eqref{gammaeq1} with  $m=1$.
        \textbf{(D)} Strong gamma distribution kernel defined in \eqref{gammaeq1} with  $m=2$. 
         }
      \label{fig_kernel_distribution}
\end{figure}

\section{Model Analysis}\label{sec:analy}
The row sum constraint \eqref{rowsum} assures that the system \eqref{eq1} admits synchronous solutions, $(E_i(t),I_i(t),W^{EI}(t))=(E(t),I(t),W^{EI}(t)),\ i=1,2,\ldots, N.$ It can be show that these solutions form an invariant subspace of \eqref{eq1}, governed by the three dimensional system 
\begin{eqnarray}\label{syncsystem}
\begin{array}{rcl}
		\displaystyle\tau_1\frac{dE}{dt} & = & -E+\phi (W^{E} \int_0^{\infty} E(t-s) g(s)\,ds- W^{EI}I ) \,, \\[2mm]
		\displaystyle\frac{dI}{dt} & = &   -I+\phi(W^{IE}E), \\[2mm]
		\displaystyle\tau_2\frac{dW^{EI}}{dt} & = & I(E-p ).
		\end{array}
\end{eqnarray}
Note that this corresponds to a single node with a delayed, excitatory self-connection.

Previous work \cite{al2021impact,nicola2021,nicola2018chaos} studied the bifurcations and dynamics of system \eqref{syncsystem}. For the case of no delay ($g(u)=\delta(u)$),  the authors in \cite{nicola2018chaos} showed that complex dynamics, including mixed mode solutions and chaos, can arise in \eqref{syncsystem} due to the slow timescale associated with the inhibitory synaptic plasticity, i.e., $\tau_2>>\tau_1$. Further, they showed
this complex dynamics is initiated at a supercritical Hopf, which was expressed as a curve $W^E_{Hopf}(W^{IE})$. Finally, they gave evidence 
that the complex dynamics of the synchronous solution gives rise
to complex dynamics in the whole network.
In \cite{nicola2021}, the authors showed that desynchronization of the nodes was more likely in large networks, and for parameter values where complex dynamics of the synchronized solution occur. 
  In the case of discrete delay, $g(u)=\delta(u-\tau)$, recent research \cite{al2021impact} has shown that even small delays can significantly impact the dynamics of the solution of \eqref{syncsystem}. Specifically, perturbation analysis shows that the Hopf curve in the parameter space $(W^{IE},W^E)$ is shifted upwards with the introduction of delay, thus stronger global excitation is needed to induce oscillations. 
  For small delays, the system dynamics are similar to the non-delay case, but for larger delays, the shift of the Hopf curve eliminates all mixed-mode, period doubled, and chaotic solutions. Based on the results of \cite{nicola2021} this elimination of complex dynamics in the synchronous subspace might be expected to lead to synchronization of the network with increasing delay. The numerical studies of \cite{al2021impact} show this is only the case for networks where the connectivity matrix has real (or close to real) eigenvalues. For networks with complex eigenvalues, increasing the delay could  {\em desynchronize} networks that were synchronized with smaller (or no) delay.

To gain understanding of this dicotomy, here we focus on analyzing the full system \eqref{eq1}.
To start, note that the full system has the equilibrium $(E_i,I_i,W^{EI}_i)=(E^*,I^*,W^{EI*}), \, i=1,2,\dots N$, as defined by 
\begin{eqnarray}\label{eq:eqmdef}
E^*=p, \, I^*=\phi(W^{IE}p), \, W^{EI*}=(W^Ep-\phi^{-1}(p))/I^*
\end{eqnarray}
We will call this the {\em synchronous equilibrium}.

Let $\e_i=E_i-E^*,\ \i_i=I_i-I^*,\ \w_i=W^{EI}_i-W_s^{EI*}$ and $\sum_{j=1}^{N}W_{ij}^{EE} = W^E$. Then the linearization about the synchronous equilibrium is
\begin{equation}\begin{array}{rcl}
\dss\tau_1 \frac{d\e_i}{dt}&=& -\e_i +
K_{1}\left(\sum_{j=1}^NW_{ij}^{EE}\int_0^\infty \e_j(t-s)g(s)\,ds -W^{EI*}\i_i-I^*\w_i \right)\\[2mm]
\dss \frac{d\i_i}{dt}&=& -\i_i +K_{2} \e_i\\[2mm]
\dss\tau_2 \frac{d\w_i}{dt}&=&I^*\e_i
\end{array}\label{linearization}\end{equation}
where  $K_{1}=\phi'(W^E E^*-{W}^{EI*}{I^*})=\phi'(\phi^{-1}(p))$,
$K_{2}=W^{IE}\phi'(W^{IE}{E}^*)=W^{IE}\phi'(W^{IE}p)$. Since $\phi$ is sigmoidal and $W^{IE}>0$, $K_1$ and $K_2$ are both positive. For the nonlinearity
$\phi(x)=1/(1+e^{-ax})$, we have $\phi'(x)=a\phi(x)(1-\phi(x))$, which gives
\begin{equation}
K_1=ap(1-p),\qquad K_2=aW^{IE}I^*(1-I^*). \label{Kdef}
\end{equation}
Let ${\bf \e}=(\e_1,\ldots,\e_N)^T$, ${\bf \i}=(\i_1,\ldots,\i_N)^T$, ${\bf \w}=(\w_1,\ldots,\w_N)^T$.
The linearization can be written in vector form as
\begin{equation}\begin{array}{rcl}
\dss\tau_1{\bf \e}'(t)&=& -{\bf \e}(t) -K_{1}\left( W^{EI*}{\bf \i}(t)+I^*{\bf \w}(t)\right)+{\bf W}^{EE}K_1\int_0^\infty {\bf \e}(t-s)g(s)\,ds\\ 
{\bf \i}'(t)&=&-{\bf \i}(t) +K_2\,{\bf \e}(t)\\
\dss\tau_2 {\bf \w}'(t)&=&I^*{\bf \e}(t)
\end{array} \label{vectorlinearization}
\end{equation}

Let ${\bf W}^{EE}$ have eigenvalues
$\hat{r}_k\in\mathbb{C}$ with corresponding eigenvectors
${\bf v}_k,\ k=1,\ldots N$. Assuming the eigenvectors are linearly independent, ${\bf W}^{EE}$
is diagonalizable. In particular, ${\bf P}^{-1}{\bf W}^{EE}{\bf P}= \diag(\hat{r}_1,\hat{r}_2,\ldots, \hat{r}_N)$, where 
 ${\bf P}=[{\bf v}_1,{\bf v}_2,\ldots,{\bf v}_N]$.

Define
\[{\bf \eta_{\e}}={\bf P}^{-1}{\bf \e},\ {\bf \eta_{\i}}={\bf P}^{-1}{\bf \i},\ {\bf \eta_{\w}}={\bf P}^{-1}{\bf \w}. \]
Then the linearization becomes
\begin{eqnarray*}
\dss\tau_1{\bf \eta_\e}'(t)
&=& -{\bf \eta_\e}(t)-K_1\left(W^{EI*}{\bf \eta_\i}(t)+I^*{\bf \eta_\w}(t)\right) +{\bf P}^{-1}{\bf W}^{EE}{\bf P}K_1\int_0^\infty  {\bf \eta_\e}(t-s)g(s)\,ds\\ 
{\bf \eta_\i}'(t)&=&-{\bf \eta_\i}(t) +K_2\,{\bf \eta_\e}(t)\\
\dss\tau_2 {\bf \eta_\w}'(t)&=&I^*{\bf \eta_\e}(t)
\end{eqnarray*}
Since ${\bf P}^{-1}{\bf W}^{EE}{\bf P}$ is diagonal, this breaks up into $N$ independent, $3D$ systems
\begin{equation}\begin{array}{rcl}
\dss\tau_1 \frac{d\eta_\e}{dt}&=& -\eta_\e 
-K_1I^*\eta_\w-K_1W^{EI*}\eta_\i+
K_1\hat{r}_k\int_0^\infty \eta_\e(t-s)g(s)\,ds 
\\[2mm]
\dss \frac{d\eta_\i}{dt}&=& -\eta_\i +K_2 \eta_\e\\[2mm]
\dss\tau_2 \frac{d\eta_\w}{dt}&=&I^*\eta_\e
\end{array}\label{componentlin}\end{equation}
where  $\hat{r}_k,\ k=1,\ldots,N$ are the eigenvalues of the connectivity matrix ${\bf W}^{EE}$.

The constant row sum constraint \eqref{rowsum} means the connectivity matrix can be written 
$\bW^{EE}=W^E {\bf L}^{EE}$ where ${\bf L}^{EE}$ has constant row sum $1$.  
It follows that each eigenvalue of $\bW^{EE}$ can be written $\hat{r}_k=W^E r_k$, where
$r_k$ is an eigenvalue of $\bL^{EE}$, and the $k^{th}$ subsystem of the linearization \eqref{componentlin} can be rewritten
\begin{equation}\begin{array}{rcl}
\dss\tau_1 \frac{d\eta_\e}{dt}&=& -\eta_\e 
-K_1I^*\eta_\w-K_1W^{EI*}\eta_\i
+K_1W^E{r}_k\int_0^\infty \eta_\e(t-s)g(s)\,ds 
\\[2mm]
\dss \frac{d\eta_\i}{dt}&=& -\eta_\i +K_2 \eta_\e\\[2mm]
\dss\tau_2 \frac{d\eta_\w}{dt}&=&I^*\eta_\e
\end{array}\label{componentlinb}\end{equation}
In vector form this system can be expressed as 
 \begin{eqnarray}\label{eq8}
\frac{d\mathbf{u}(t)}{dt}=\bL_0\mathbf{u}(t) + W^E r_k\bM \int_0^{\infty} \mathbf{u}(t-s) g(s)\,ds,
\end{eqnarray}
where
\begin{eqnarray}\label{matrixLoM}
\bL_0=\left(
\begin{array}{cccc}
	-\frac{1}{\tau_1}& -\frac{1}{\tau_1}K_1W^{EI*} & -\frac{1}{\tau_1}K_1I^* \\
	K_2 & -1 & 0  \\
	\frac{1}{\tau_2}I^* & 0 & 0   
\end{array}
\right), \quad
\bM=\left(
\begin{array}{cccc}
	\frac{1}{\tau_1}K_1& 0& 0 \\
	0 & 0 & 0  \\
	0& 0 & 0   
\end{array}
\right)
\end{eqnarray}

The characteristic matrix of this system is 
$\det[\lambda \bI-\bL_0-r_kW^E\hat{G}(\lambda)\bM]$,
where $\bI$ is the $3\times 3$ identity matrix. It follows that the characteristic equation for the $k^{th}$ subsystem is
  \begin{eqnarray}\label{charac_k}
 C_k(\lambda)=P(\lambda)-r_kQ(\lambda)=\lambda^3+p_2\lambda^2+p_1\lambda+p_0- r_k W^E q\lambda(\lambda+1)\hat{G}(\lambda)=0
  \end{eqnarray}
  where 
\begin{neweq}
p_2&=\cfrac{1}{\tau_1}+1,\quad
p_1=\cfrac{1}{\tau_1}+\cfrac{{W}^{EI*}K_1K_2}{\tau_1} +\cfrac{{I^*}^2K_1}{\tau_1\tau_2},\quad 
p_0 =\cfrac{{I^*}^2K_1}{\tau_1\tau_2}, \quad
q=\cfrac{K_1}{\tau_1}
\label{eq:pdef}
\end{neweq}
 and 
\[\widehat{G}(\lambda)=\int_0^{\infty} e^{-\lambda s}g(s)ds\]
is the Laplace transform of the distribution kernel $g$. Recall that $K_1=\phi'(\phi^{-1}(p)),\  K_2=W^{IE}\phi'(W^{IE}p)$ and the equilibrium values $I^*,W^{EI*}$ are defined in \eqref{eq:eqmdef}.
Since $K_1,K_2,\tau_1,\tau_2$ are positive it follows that $p_0,p_2$ and $q$ are all positive. We only consider $W^{EI*}\ge 0$ for biological realism, thus $p_{1}$ is also positive.

To summarize, the characteristic equation of the linearization of the whole network about the synchronous equilibrium is
\begin{equation}  \prod_{k=1}^N C_k(\lambda)=0. \label{charac} \end{equation}
where $C_k$ is defined in \eqref{charac_k}.

Now, since $\bL^{EE}$ has row sum $1$, it has $r_N=1$ and $|r_k|\le 1,\ k=1,\ldots, N-1$, by the Perron-Frobenius Theorem. The subsystem with $k=N$, corresponds to the linearization of the synchronous equilibrium in the synchronous subspace, i.e., eq. \eqref{syncsystem}. Thus bifurcations associated with the factor \eqref{charac_k} with $r_k=1$ will yield bifurcations of synchronous solutions. The associated Hopf bifurcation, which leads to synchronous oscillations in the network, was studied in depth by \cite{al2021impact,nicola2021}. We will call this the {\em synchronous Hopf bifurcation}.
 We are interested in Hopf bifurcations coming from the other subsystems, corresponding to other eigenvalues of $\bL^{EE}$: $r_1,r_2\,\ldots,r_{N-1}$. Our hypothesis
is that these bifurcations, or their interactions with the synchronous Hopf, can give rise to desynchronized solutions.

In \cite{al2021impact,nicola2021} several different connectivity matrices were studied. In the following, we focus on the case where $\bL^{EE}$ is circulant, which includes the bi-directional and uni-directional rings studied by \cite{al2021impact,nicola2021} and will allow us to determine the synchronization pattern of the solutions arising from bifurcations corresponding to other factors of the characteristic equation.

If the matrix $\bL^{EE}$ is circulant, we can leverage the work of \cite{wang2017}, which is formulated for systems in a slightly different form than \eqref{eq1}, but which can be carried over here as we now show.
 Let ${\bf L}^{EE}={\rm circ}\{0,l_1,\cdots,l_{N-1}\}$  where $l_j\in\mathbb R$ and  the first component is zero since we assume no self-coupling.  The eigenvalues of this matrix are 
\[ r_k=\sum_{j=0}^{N-1} l_j \rho_k^{j},\quad
 k=1,\ldots,N\]
where $\rho_k=e^{i\frac{2\pi k}{N}}$ are the $N^{th}$ roots of unity. The corresponding eigenvectors are 
\[ {\bf v}_k=\begin{pmatrix} 1\\ \rho_k\\ \rho_k^2\\ \vdots \\ \rho_k^{N-1} \end{pmatrix}.\]
\begin{proposition}\label{prop_evector}
If ${\bf L}^{EE}$ is circulant and $\lambda$ is root of \eqref{charac_k}, 
then the corresponding solution of \eqref{vectorlinearization} is 
\begin{equation}
     \begin{bmatrix}{\bf x} \\ {\bf y}\\ {\bf z} \end{bmatrix}= e^{\lambda t}\begin{bmatrix} \xi_x{\bf v}_k\\ \xi_y{\bf v}_k \\ \xi_z{\bf v}_k \end{bmatrix},
 \label{vecsoln}
 \end{equation}
where ${\bf v}_k$ is the eigenvector of ${\bf L}^{EE}$ corresponding to the eigenvalue ${r}_k$,
and ${\bf \xi}=(\xi_x,\xi_y,\xi_z)^T$ satisfies 
\[\left[\lambda \bI - \bL_0 - r_k W^E\widehat{G}(\lambda) \bM\right]\xi={\bf 0}. \]
\end{proposition}
\begin{proof}
We follow the idea of the proof in section 1 of \cite{wang2017}.
Let ${\bf X_i}=(x_i,y_i,z_i)^T$. Then \eqref{vectorlinearization} can be written
\begin{equation}
{\bX_i}'(t)={\bf L}_0{\bX_i}+\sum_{j=1}^{N}W^{EE}_{ij} \bM\int_{0}^\infty {\bX}_j(t-s)g(s)\, ds.  
\label{linnewform}
\end{equation}

Note that the solution \eqref{vecsoln} can be expressed in these
new coordinates as $\bX_i=\Psi_i(t)$ where
\begin{equation}
\Psi_i(t)=e^{\lambda t} \bv_{k,i}\, \xi=e^{\lambda t} \rho_k^{i-1} \xi
    \label{solnnewform}
\end{equation}
Substituting this into the RHS of \eqref{linnewform} gives
\begin{eqnarray*} \bL_0 e^{\lambda t}\bv_{k,i}\, \xi + \sum_{j=1}^{N}W^{EE}_{ij} \bv_{k,j} \bM \xi e^{\lambda t}\int_{0}^\infty e^{\lambda s}g(s)\, ds. 
&=& e^{\lambda t} \left[\bL_0\bv_{k,i}\, \xi + W^E\sum_{j=1}^{N}L^{EE}_{ij}\bv_{k,j} \bM \xi \widehat{G}(\lambda)\right]\\
&=& e^{\lambda t} \left[\bL_0 \rho_k^{i-1} \xi + W^E\widehat{G}(\lambda)\left({\bf L}^{EE}{\bf v}_k\right)_i \bM \xi \right]\\
&=& e^{\lambda t} \bv_{k,i} \left[\bL_0  + W^E r_k \widehat{G}(\lambda) \bM \right]\xi \\
&=& \lambda e^{\lambda t} \bv_{k,i}\,\xi
\end{eqnarray*}
Thus $\bX_i=\Psi_i(t)$ is a solution of \eqref{linnewform} and
hence \eqref{vecsoln} is a solution of \eqref{vectorlinearization}.
\end{proof}
This proposition can be used to predict the structure of solutions associated with bifurcations. For example, if $k=N$ ($r_k=1$), then ${\bf v}_k=(1,1,\ldots,1)^T$, which verifies that the solutions in this subsystem correspond to synchronous solutions. Similarly, if $k=1$ then ${\bf v}_k=(1,\rho_1,\rho_2,\ldots,\rho_{N-1})^T$, i.e., the entries are the $N^{th}$ roots of unity. If we visualize the nodes on a ring, then this solution will be such that the nodes are equally spaced around the ring. This is sometimes called a splay state or, in the case of periodic solutions, a travelling wave solution. Equivariant bifurcation theory \cite{GSS88,Wu98} shows that this structure observed in the linearization carries over to bifurcating solutions. So if a Hopf bifurcation occurs corresponding to an eigenvalue for the subsystem $k=N$, then the resulting solutions are periodic orbits with all the nodes oscillating synchronously, while if a Hopf bifurcation occurs corresponding an eigenvalue for the subsystem $k=1$, the resulting solutions are periodic orbits with the nodes phase-locked with phase difference $1/N$ of the period. More detail about the application of equivariant bifurcation theory to the study of phase-locked solutions in coupled cell systems with time delay can be found in \cite{GH03,GH07,wang2017,Wu98}.

Recall that since ${\bf L}^{EE}$ has row sum $1$, it always has the eigenvalue $1$. In the circulant case, we see directly that $r_N=\sum_{k=0}^{N-1} l_k=1$. Other eigenvalues come in complex conjugate pairs: $r_{N-j}=\bar{r}_j$, except for special cases. For example, if $N$ is even then  $r_{N/2}=\sum_{k=0}^{N-1} (-1)^k l_k$. We can write the characteristic equation in the form
\begin{eqnarray*}
0=C(\lambda)&=&C_{N}(\lambda)C_{N/2}(\lambda)
\prod_{k=1}^{\floor*{N/2-1}} C_{k}(\lambda)C_{N-k}(\lambda).\\
\end{eqnarray*}
Since $C_N(\lambda)$ and $C_{N/2}(\lambda)$ have real coefficients, if $\lambda$ is a root of either factor then so is $\bar{\lambda}$. The other $C_k(\lambda)$ in general have complex roots with $C(\lambda)=0$ implying 
$C_{N-k}(\bar{\lambda})=0.$

We will focus the special case of uni-directional and bidirectional rings as this will give examples of the two cases we want to investigate, real and complex eigenvalues. 
In the case of a unidirectional ring, 
${\bf L}^{EE}={\rm circ}\{0,1,\cdots,0,0\}$ and the eigenvalues are
$ r_j= \rho_j^{N-1} = e^{\frac{2\pi ij}{N}}$, which are in general complex. 
If $\bL^{EE}$ is circulant and symmetric 
(e.g.~the bidirectional ring or lattice) then
${\bf L}^{EE}={\rm circ}\{0,l_1,\cdots,l_2,l_1\}$  and the eigenvalues are
real. In the case of a symmetric bi-directional loop, 
${\bf L}^{EE}={\rm circ}\{0,1/2,0,\cdots,0,1/2\}$ and the eigenvalues are
$ r_j=\frac12(\rho_j+\rho_j^{N-1}) = \frac12(e^{\frac{2\pi ij}{N}}+ e^{-\frac{2\pi ij}{N}})
=\cos(\frac{2\pi j}{N})$.  Note that $r_{N-j}=r_j$.

\subsection{Stability}
To begin we establish some parameter values where we know the synchronous equilibrium point is asymptotically stable.
\begin{theorem}\label{thm:stabnodelay}
If $g(s)=\delta(s)$ and
   $ W^EK_1 <1$
then the equilibrium $(p,I^*,W^{EI*})$ is asymptotically stable.
\end{theorem}
\begin{proof}
The proof can be found in Appendix \ref{appendix_stabnodelay}. 
\end{proof}
\begin{theorem}\label{thm:stabsmalldelay}
If  $g(s)$ is a Dirac delta, uniform, weak or strong gamma distribution and $ W^EK_1 <1$ then the equilibrium $(p,I^*,W^{EI*})$ is asymptotically stable for sufficiently small mean delay $\tau_m$.
\end{theorem}
\begin{proof}
The proof can be found in Appendix \ref{appendix_stabsmalldelay}.   
\end{proof}

In \cite{al2021impact,nicola2021,nicola2018chaos} the loss of stability of the equilibrium was studied in the case of a single self-coupled node, which corresponds to studying stability within the synchronous subspace.  By studying \eqref{charac_k} with $r_k=1$, i.e., $C_N(\lambda)=0$, they showed that the equilibrium is asymptotically stable with $W^E<1/K_1$ and loses stability in a Hopf bifurcation as $W^E$ is increased. Here we consider the possibility of destabilization outside of this subspace. Since $p_0>0$, the equation \eqref{charac_k} has no zero roots for any $k$, destabilization can only occur if some factor of the  the characteristic equation has a pure imaginary eigenvalue. In the following we characterize the parameter values where this occurs.

\subsection{Curves of Pure Imaginary Eigenvalues}\label{ssec:Hcurves}
Suppose that 
$\lambda=i \omega$ ($\omega>0$ and $i=\sqrt{-1}$) is a purely imaginary root of \eqref{charac_k}, where $r_k=\alpha+i\beta$. We can restrict our attention to $\omega>0$ since  $\lambda=-i\omega$ will be a root of $C_k(\lambda)$ if $\beta=0$ and $\lambda=-i\omega$ will be a root of $C_{N-k}(\lambda)$ of $\beta\ne 0$.

 The real and imaginary parts of ${C}_k(i\omega)=0$ satisfy 
\begin{align}
p_0-p_2 \omega^2&=W^E q\,\,\omega  \Big( (\alpha-\beta\omega) S(\omega)-(\alpha\omega+\beta) C(\omega)\Big)\label{eq:ReIm_1e} \\
p_1\omega-\omega^3&= W^E q\, \, \omega  \Big((\alpha\omega+\beta) S(\omega)+ (\alpha-\beta\omega) C(\omega)\Big)\label{eq:ReIm_2e}.
\end{align}
where 
\begin{equation}\label{sincos}
   S(\omega)=\int_{0}^{\infty}\sin \! \left(\omega  s \right) g \! \left(s \right)d s \quad\text{and}\quad
C(\omega)=\int_{0}^{\infty}\cos \! \left(\omega  s \right) g \! \left(s \right)d s. 
\end{equation}
Further, note from \eqref{eq:pdef} that $p_1=p_{1a}+p_{1b} W^E q$ where
\begin{align}
 p_{1a}&=\frac{1}{\tau_1}\left[1+\frac{I^{*2}K_1}{\tau_2}-\frac{\phi^{-1}(p)K_1K_2}{I^*}\right]
=p_0+\frac{1}{\tau_1}\Big[1-\frac{\phi^{-1}(p)\,\phi'(\phi^{-1}(p))\,W^{IE}\phi'(W^{IE}p)}{I^*}\Big] \label{eq:p1a}\\ p_{1b}&=\frac{pK_2}{I^*}=\frac{p\,W^{IE}\phi'(W^{IE}p)}{I^*}. \label{eq:p1b}
\end{align}
With the logistic nonlinearity, $\phi(x)=(1+e^{-ax})^{-1}$, these simplify to
\begin{align}
 p_{1a}&=p_0+\frac{1}{\tau_1}\Big[1+(1-p)\ln\left(\frac{p}{1-p}\right)apW^{IE}(1-I^*)\Big] \label{eq:p1ab}\\ p_{1b}&=apW^{IE}(1-I^*), \label{eq:p1bb}
\end{align}
from which it can be shown that
$0\le p_{1b}<0.2785$ and $p_{1a}>p_0$.

We can use equations \eqref{eq:ReIm_1e}-\eqref{eq:ReIm_2e} to characterize the curves of pure imaginary eigenvalues in terms of the connection weights $W^E$ and $W^{IE}$. We can then study the effect of the delay and other parameters on these curves. These curves may be associated with a Hopf bifurcation of the synchronous equilibrium point. Due to the complexity of the expressions, we do not attempt to verify the conditions of the Hopf bifurcation Theorem analytically. In the next section we will verify using numerical bifurcation software that Hopf bifurcations do occur along these curves. Thus we will refer to the curves of pure imaginary eigenvalues as Hopf curves.

Solving \eqref{eq:ReIm_1e} for $W^{E}$ and substituting into \eqref{eq:ReIm_2e} we have
\begin{equation}
    W^E_{Hopf}(W^{IE})=\frac{p_2\omega^2-p_0}{q\omega\Big((\alpha\omega+\beta)C(\omega)-(\alpha-\beta\omega)S(\omega)\Big)} \label{eq:WEhopf}
\end{equation}
where $\omega$ is a root of
\begin{align}
&-\Big(C(\omega)\alpha + S(\omega)\beta\Big)\omega^4 + (p_2 - 1)\Big(C(\omega)\beta - S(\omega)\alpha\Big)\omega^3 +  \Big(\left(C(\omega)\alpha+S(\omega)\beta\right)( p_{1a} - p_2) + p_{1b}p_2\Big)\omega^2  \nonumber\\
&\hspace{6cm} - (p_0 - p_{1a})\Big(C(\omega)\beta - S(\omega)\alpha\Big)\omega + p_0\Big(C(\omega)\alpha + S(\omega)\beta - p_{1b}\Big)=0
  \label{eq:omega}\end{align}
Recalling the definitions of the $p_j$, \eqref{eq:pdef}, we that $W^E_{Hopf}$ depends on $W^{IE}$ directly through $p_0$ and indirectly through $\omega$.
It is difficult to analyze these equations in general, but we can consider some limiting behaviour. 

\begin{proposition}\label{thm:endpt}
If $\beta=0$ the endpoints $W^E_{Hopf}(0)$ and $\lim_{W^{IE}\rightarrow \infty} W^E_{Hopf}(W^{IE})$ are decreasing functions of $\alpha$.
\end{proposition}
\begin{proof}
When $W^{IE}=0$, from \eqref{eq:eqmdef}, \eqref{Kdef} and \eqref{eq:pdef}  we have 
\[ I^*=\phi(0)=\frac12,\ K_2=0, \ p_0=\frac{K_1}{4\tau_1\tau_2},\ p_{1a}=\frac{1}{\tau_1}+p_0,\ p_{1b}=0,\ p_{2}=1+\frac{1}{\tau_1}. \]
With these limits we see that
\begin{equation}
    W^E_{Hopf}(0)=\frac{p_2\omega^2-p_0}{\alpha q\omega\Big(\omega C(\omega)-S(\omega)\Big)} \label{eq:WEhopf0}
\end{equation}
where $\omega$ is a root of
\begin{equation}
-C(\omega)\omega^4 - (p_2 - 1)S(\omega)\omega^3 +  \Big(C(\omega)( p_{1a} - p_2)\Big)\omega^2  
+ (p_0 - p_{1a})S(\omega)\omega + p_0 C(\omega)=0.
  \label{eq:omega0}\end{equation}
Since \eqref{eq:omega0} does not depend on $\alpha$, it is clear from \eqref{eq:WEhopf0} that $W^E_{Hopf}(0)$ is a decreasing function of $\alpha$.
  
When $W^{IE}\rightarrow\infty$ we have
\[ I^*\rightarrow 1,\ K_2\rightarrow 0, \ p_0=\frac{K_1}{\tau_1\tau_2},\ p_{1a}\rightarrow \frac{1}{\tau_1}+p_0,\ p_{1b}\rightarrow 0,\ p_{2}=1+\frac{1}{\tau_1}. \]
Thus the expression for $W^E_{Hopf}(W^{IE})$ in this case is the same as for $W^E_{Hopf}(0)$, with just a different value for $p_0$. Hence this is also a decreasing function of $\alpha$
\end{proof}
If all the eigenvalues are real, then they satisfy $\alpha\le 1$ and this theorem implies that, when $W^{IE}=0$ or is arbitrarily large, the Hopf bifurcation point corresponding to $r_k=1$ has the lowest value of $W^E$. It follows from the discussion following Proposition~\ref{prop_evector}, that the equilibrium point will lose stability via a Hopf bifurcation giving rise to an oscillation where all the nodes are synchronized.  To consider what happens when $W^{IE}$ between these limits and for connection matrices with complex eigenvalues, we consider some specific distributions.

\subsubsection{No delay} \label{sec:curvesnodelay}
The case of zero delay corresponds to the delay distribution being given by a Dirac delta, $g(s)=\delta(s)$. In this case the $k^{th}$ factor of the characteristic equation \eqref{charac_k} is 
a cubic polynomial in $\lambda$:
\begin{equation}\label{charac_eq_ODE}
    \lambda^3+p_2\lambda^2+p_1\lambda+p_0-r_k W^Eq\lambda(\lambda+1)=0. 
\end{equation}
In this case $C(\omega)=1$ and $S(\omega)=0$ and the curves of pure imaginary eigenvalues are given by
\begin{equation}
W^E_{Hopf}(W^{IE})=\frac{p_2\omega^2-p_0}{q(\omega^2\alpha+\omega\beta)} \label{eq:WEhopfnodelay}
\end{equation}
where $\omega$ is a root of
\begin{equation}
\alpha\omega^4 - (p_2 - 1)\beta\omega^3 -  \Big(\alpha( p_{1a} - p_2) + p_{1b}p_2\Big)\omega^2+ (p_0 - p_{1a})\beta \omega + p_0\Big( p_{1b}-\alpha\Big)=0.
\label{eq:omeganodelay}\end{equation}

In Appendix~\ref{appendix_C} we show that the following are satisfied. 
\begin{equation} \frac{1}{\alpha K_1}\le W^E_{Hopf}(W^{IE}),\label{curveorder} \end{equation}
for any $\r_k=\alpha+i\beta$, and, if  $r_k=\alpha\in{\mathbb R}$,  then $W^E_{Hopf}$ is a decreasing function of $\alpha$
with 
\begin{equation} \frac{1}{\alpha K_1}\le W^E_{Hopf}(W^{IE})<\frac{1+\tau_1}{\alpha K_1}.\label{curveorderreal}\end{equation}
The equality only holds when $W^{IE}=0$ and $W^{IE}\rightarrow\infty$.

These results tell us about the geometry of the curves of pure imaginary eigenvalues. First the endpoints of each curve are determined by the real part of the corresponding eigenvalue of the connectivity matrix $r_k$. Thus, these endpoints are ordered by the size of $\Re(r_k)$. This is a slight generalization of Proposition~\ref{thm:endpt}. The main differences are that we can show this for any $r_k\in{\mathbb C}$ and we have an exact expression for the value of $W^E$ at the endpoints. Further, \eqref{curveorder} indicates that each curve is bounded below by the values at the endpoints. If all the $r_k$ are real, we have also found an upper bound for each curves given by \eqref{curveorderreal}. 

Other results in Appendix~\ref{appendix_C} for the case $r_k=\alpha\in{\mathbb R}$ are as follows.
\begin{itemize}
\item If $\alpha<0$ eq.~\eqref{charac_eq_ODE} only has no pure imaginary eigenvalues 
\item For a given $\alpha>0$, there is only one curve of pure imaginary eigenvalues given by \eqref{eq:WEhopfnodelay}-\eqref{eq:omeganodelay}
\item  $W^E_{Hopf}(W^{IE})$ as given by \eqref{eq:WEhopfnodelay} is a decreasing function of $\alpha$
\end{itemize}
These results have the strongest implication if all the $r_k$ are real. In this case, $\alpha\le 1$ and the Hopf bifurcation curve corresponding to $r_k=1$ lies below all the other Hopf curves. The discussion following Proposition~\ref{prop_evector} then shows that, for any value of $W^{IE}$, the equilibrium point will lose stability via a Hopf bifurcation giving rise to an oscillation where all the nodes are synchronized. The Hopf bifurcations corresponding to other $\alpha$ are not expected to affect the observable synchronization of the nodes as they will occur when the equilibrium point is already unstable.

\subsubsection{Gamma distribution function}
Recall that this distribution is given by \eqref{gammaeq1}.
The Laplace transform of this distribution is
\begin{equation}\label{gammaeq3}
\widehat{G}(\lambda)=\left(\frac{\gamma}{\lambda+\gamma}\right)^m.
\end{equation}
Thus the $k^{th}$ factor of the characteristic equation \eqref{charac_k} becomes
  \begin{eqnarray}\label{gammaeq4}
  \lambda^3+p_2\lambda^2+p_1\lambda+p_0-r_kqW^E\lambda(\lambda+1)\left(\frac{\gamma}{\lambda+\gamma}\right)=0
  \end{eqnarray}
  
Recall the {\em weak} gamma distribution corresponds to $m=1$ and has mean delay $\tau_m=\frac{1}{\gamma}$. Thus the $k^{th}$ factor of the characteristic equation for this case can be written
  \begin{eqnarray}\label{gammaeq4A}
  \lambda^3+p_2\lambda^2+p_1\lambda+p_0-r_kqW^E\lambda(\lambda+1)\left(\frac{1}{\lambda\tau_m+1}\right)=0,
  \end{eqnarray}
and the curves of pure imaginary eigenvalues are described by \eqref{eq:WEhopf}-\eqref{eq:omega} with 
\[
C(\omega)=\frac{1}{\omega^2\tau_m^2+1},\quad S(\omega)=\frac{\omega\tau_m}{\omega^2\tau_m^2+1}.
\]

The {\it strong gamma distribution}, corresponds to $m=2$, 
and has mean delay  $\tau_m=\frac{2}{\gamma}$. Thus equation \eqref{charac_k} can be written
 \begin{eqnarray}\label{gammaeq5}
  \lambda^3+p_2\lambda^2+p_1\lambda+p_0-r_kqW^E\lambda(\lambda+1)\left(\frac{1}{\frac12\lambda\tau_m+1}\right)^2=0,
  \end{eqnarray} 
and the curves of pure imaginary eigenvalues are described by  \eqref{eq:WEhopf}-\eqref{eq:omega} with 
\[
C(\omega)=\frac{1-\frac14\omega^2\tau_m^2}{(1+\frac14 \omega^2\tau_m^2)^2},\quad S(\omega)=\frac{\omega\tau_m}{(1+\frac14\omega^2\tau_m^2)^2}.
\]

In both cases above,  $C(\omega)$ and $S(\omega)$ are rational functions of $\omega$, hence explicit expressions can be found for  $W^E_{Hopf}(W^{IE})$. As the expressions are quite complex, we give them in Appendices~\ref{appendix_C1} and \ref{appendix_B2}.

\subsubsection{Dirac delta distribution}
This distribution is given by a Dirac delta shifted by $\gavg$, i.e., $g(s)=\delta(s-\gavg)$. Thus the $k^{th}$ factor of the characteristic equation, \eqref{charac_k}, becomes 
\[ \lambda^3+p_2\lambda^2+p_1\lambda+p_0- r_k qW^E\lambda(\lambda+1)e^{-\lambda\gavg}=0, \]
and the curves of pure imaginary eigenvalues are described by \eqref{eq:WEhopf}-\eqref{eq:omega} with $C(\omega)=\cos(\omega\gavg)$ and
$S(\omega)=\sin(\omega\gavg)$. 
Hence \eqref{eq:omega} is no longer a polynomial equation that can be explicitly solved for $\omega$. In \cite{al2021impact}, for the case $r_k=1$ and $\tau_m$ small, they derive  perturbation expansion
$W^{E}_{Hopf}(W^{IE})=W^{E}_{Hopf,0}(W^{IE})+W^{E}_{Hopf,1}(W^{IE})\tau_m+O(\tau_m^2)$.  While this approach can be easily extended to the situation when $r_k$ is any real number, when $r_k$ is complex it is more complicated. Thus we will instead use a numerical approach to study the Hopf curves in this case.

\subsubsection{Uniform distribution function}
Recall that this distribution is given by \eqref{unieq1}. The Laplace transform of this distribution is
  \begin{equation}\label{unieq3}
  \hat{G}(\lambda)=\frac{1}{2\sigma\lambda}e^{-\lambda\gavg}(e^{\lambda\sigma}-e^{-\lambda\sigma})=e^{-\lambda\gavg}\sinhc (\lambda\sigma),
  \end{equation}
  where
  \[ \sinhc(u)=\left\{ \begin{array}{cc}
      \frac{\sinh (u)}{u},&\mbox{if } u\ne 0 \\
       1,& \mbox{if } u=0
  \end{array}\right.\]
  Thus, the $k^{th}$ factor of the characteristic equation, \eqref{charac_k}, becomes 
  \begin{eqnarray}\label{unieq4}
  \lambda^3+p_2\lambda^2+p_1\lambda+p_0-r_kqW^E\lambda(\lambda+1)e^{-\lambda\gavg}\sinhc(\lambda \sigma)=0,
  \end{eqnarray}
and the curves of pure imaginary eigenvalues are described by \eqref{eq:WEhopf}-\eqref{eq:omega} with 
\[
C(\omega)=\sinc(\omega\sigma)\cos(\omega\gavg) \quad \text{and}\quad S(\omega)={\sinc}(\omega\sigma)\sin(\omega\gavg).
\]
\[ 
\sinc(u)=\left\{ \begin{array}{cc}
      \frac{\sin (u)}{u},&\mbox{if } u\ne 0 \\
       1,& \mbox{if } u=0
  \end{array}\right.\]
As for the Dirac delta case, here \eqref{eq:omega} is a transcendental equation in $\omega$ and cannot be solved explicitly. Thus we will treat this case numerically.

\section{Numerical Results}\label{sec:num}
\subsection{Numerical Bifurcation Curves}\label{sec:numbif}

From the work in \cite{al2021impact,nicola2021},  we may expect that the eigenvalue(s) that causes desynchronization should be the
one(s) with largest real part, that are not the eigenvalue $1$. Thus in our numerical work focus on roots of \eqref{charac_k} with these values of $r_k$.
For the bi-directional ring this is $r_1=\cos(\frac{2\pi}{N})$, leading to the equation
\begin{equation}\label{charac_eq_bi}
 0=C_1(\lambda)
=P(\lambda)-\cos\left(\frac{2\pi}{N}\right)Q(\lambda).
\end{equation}
Since the roots of this equation come in complex conjugate pairs, we need only consider $\lambda=i\omega$ with $\omega>0.$
For the uni-directional ring the relevant eigenvalues are   $\cos(\frac{2\pi}{N})+\pm i\sin(\frac{2\pi}{N})$, corresponding to $r_1$ and $r_{N-1}$. So we need to investigate the pure imaginary roots of the factors 
\begin{equation}\label{charac_eq_uni}
  0=C_1(\lambda)
=P(\lambda)-\left(\cos\left(\frac{2\pi}{N}\right)+ i\sin\left(\frac{2\pi}{N}\right)\right)Q(\lambda)  
\end{equation}
and
\begin{equation}\label{charac_eq_uni_a}
  0=C_{N-1}(\lambda)
=P(\lambda)-\left(\cos\left(\frac{2\pi}{N}\right)- i\sin\left(\frac{2\pi}{N}\right)\right)Q(\lambda)  
\end{equation}
As noted above, if $\lambda=i\omega$ is a root of \eqref{charac_eq_uni} then $\bar{\lambda}=-i\omega$ is a root of \eqref{charac_eq_uni_a} and vice-versa. Thus we may focus on roots of \eqref{charac_eq_uni} with $\omega\in\mathbb{R}$. Note that this corresponds to $\alpha=\cos(\frac{2\pi}{N})$ and $\beta=0$ or $\beta=\sin(\frac{2\pi}{N})$ in the expressions derived for the curves $W^E_{Hopf}(W^{IE})$ derived in section~\ref{ssec:Hcurves}.

As discussed in the previous section, these curves may be associated with a Hopf bifurcation of the equilibrium point. For the case of zero delay, we have use the numerical bifurcation package Matcont\cite{dhooge2003matcont} to verify this. For the parameters we explored, the curve with $r_k=1$ corresponds to a supercritical Hopf bifurcation.  For the bidirectional ring, \eqref{charac_eq_bi} also has one curve of supercritical Hopf bifurcations. For the uni-directional ring however,
eq.~\eqref{charac_eq_uni} has two associated curves of supercritical
Hopf bifurcations. See
Figure \ref{Fig_Hopf_coefficient} in Appendix \ref{appendix_HopfHopf}. From Proposition~\ref{prop_onecurve} and the following discussion, we will expect the Hopf bifurcation associated with the $r_k=1$ curve to result in oscillations where all the nodes are synchronized, while that associated with $r_k=\cos(\frac{2\pi}{N})$ or $r_k=\cos(\frac{2\pi}{N})+\pm i\sin(\frac{2\pi}{N})$ will result in travelling wave/splay state oscillations. Thus we will refer to these as the {\em synchronous Hopf curve} and {\em asynchronous Hopf curves}, respectively.

We fix the values of the model parameters as used in \cite{nicola2018chaos,nicola2021,al2021impact} 
\begin{equation} p=0.2,\ a=5,\ \tau_1=1,\ \tau_2=5,\label{params} \end{equation}
and used a numerical approach to explore the effect of the size of the network and the size of delay on the
 the curves of pure imaginary eigenvalues in the $(W^{IE},W^E)$ parameter plane. Figure \ref{fig_hopf_Curves_epsilon_0_1} shows the curves generated using \textit{Wolfram Mathematica}. We use the command \textsf{Solve} to find the maximum real part $\hat{\lambda}$ of the roots of the characteristic equations  \eqref{charac_eq_uni} and \eqref{charac_eq_bi} in a circle of radius eight  units centered at the origin. 
To approximate the curves, we calculate $\hat{\lambda}$  on a grid over the $(W^{IE},W^E)$-plane. Then, we approximate the boundary of the regions where $\hat{\lambda}$ changes its  sign from negative to positive using the commands \textsf{ListInterpolation} and  \textsf{RegionPlot}. To simplify these plots we only include the lower of the two asynchronous Hopf curves as the upper one cannot affect the observable behaviour.
To verify the accuracy of our numerical approach we used \textit{Maple} to solve the exact equations \eqref{eq:WEhopf}-\eqref{eq:omega} for the curves derived in section~\ref{ssec:Hcurves} in the case of no delay and the weak gamma distribution. Comparison plots are shown in 
Appendix \ref{appendix_NumericalAccuracy}.

For the bi-directional ring, we observe the following behaviour, for the values of $N$ and the mean delays we explored.
The asynchronous Hopf curves (solid red, green and blue curves in Figure~\ref{fig_hopf_Curves_epsilon_0_1}) have a similar shape to, and always lie completely above the synchronous Hopf curve (black curves). The asynchronous Hopf curves are ordered according to the number of neurons in the network, as $N$ increases the asynchronous curves get closer to the synchronous curve. For the case of no delay, these results are as predicted in section~\ref{sec:curvesnodelay} since the eigenvalues are real and $\alpha$ decreases as $N$ increases. Further, we observe that as the mean delay is increased, the synchronous and asynchronous curves move upward, but maintain their ordering.  

For the uni-directional ring, we observe the following behaviour, for the values of $N$ and the mean delays we explored. The asynchronous Hopf curves (dashed red, green and blue curves in Figure~\ref{fig_hopf_Curves_epsilon_0_1} have a different shape than and may intersect the synchronous Hopf curve (black curves). The asynchronous Hopf curves are ordered according to the number of neurons in the network, as $N$ increases the asynchronous Hopf curves move downwards in the $(W^{IE},W^E)$ parameter space.  As the mean delay is increased, the asynchronous curves move downwards and maintain their ordering.
\begin{figure}[H]
     \centering
         \includegraphics[width=1\textwidth]{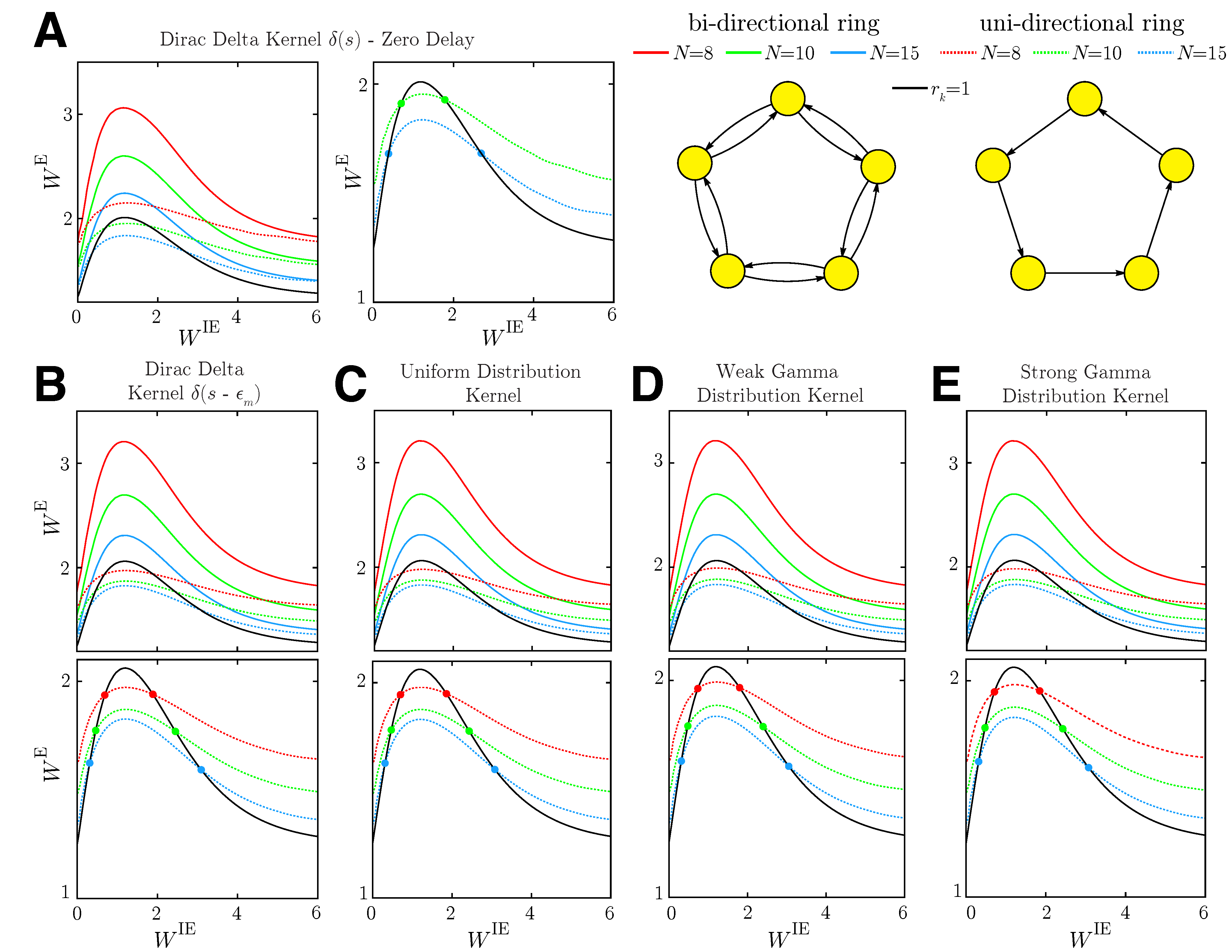}
         \caption{\textbf{Hopf curves in the $(W^{IE},W^E)$} parameter space.
         The solid (dashed) line represents bi-directional (uni-directional) ring  with $N=8$ red, $N=10$ green, and $N=15$ blue.
         \textbf{(A)} Dirac Delta kernel $g(s)=\delta(s)$ (no delay).
         Distributed delays with $\tau_m=0.1$.
         \textbf{(B)} Dirac Delta kernel $g(s)=\delta(s-\gavg)$ (fixed delay).
        \textbf{(C)} Uniform distribution kernel with $\sigma=0.1$.
        \textbf{(D)} Weak gamma distribution kernel with $\gamma=10$.
        \textbf{(E)} Strong gamma distribution kernel with $\gamma=20$. 
         }
         \label{fig_hopf_Curves_epsilon_0_1}
\end{figure}
\begin{table}[H]
\begin{tabular}{|c|c|c|c|}
\hline
Delay & $N=8$ & $N=10$ & $N=15$\\
\hline
none & none &  $(0.691,1.911)$ and $(1.780,1.928)$ & $(0.373,1.681)$ and $(2.688,1.683)$ \\
uniform & $(0.702, 1.944)$ and $(1.861, 1.949)$ & $(0.472, 1.780)$ and $(2.433, 1.775)$ & $(0.314, 1.626)$ and $(3.079, 1.597)$\\
weak & $(0.736, 1.963)$ and $(1.800, 1.967)$ & $(0.482, 1.790)$ and $(2.398, 1.787)$  & $(0.316, 1.628)$ and $(3.051, 1.604)$\\
strong & $(0.710, 1.948)$ and $(1.846, 1.953)$ &  $(0.474, 1.783)$ and $(2.424, 1.778)$ & $(0.314, 1.626)$ and $(3.072, 1.599)$\\
\hline
\end{tabular}
\caption{Location of intersection points of the synchronous and asynchronous Hopf bifurcation curves shown in Figure~\ref{fig_hopf_Curves_epsilon_0_1}.}
\label{tab:intpts}
\end{table}

The location of the asynchronous Hopf curve in the bi-directional ring means that on these curves, the synchronous equilibrium point is unstable hence crossing these curves is not expected to change the observable behaviour of the system. As $W^E$ is increased from small values, we expect the synchronous limit cycle to appear as the synchronous Hopf curve is crossed and no further changes in behaviour will be associated with bifurcations of the equilibrium point.

The location of the asynchronous Hopf curves in the uni-directional ring case,  means that for sufficiently large $N$ with fixed mean delay, or sufficiently large mean delay with fixed $N$ the lower asynchronous Hopf curve intersects the synchronous Hopf curve in two places, $(W^{IE}_L,W^{E}_L)$ and $(W^{IE}_R,W^E_R)$. Here we focused on relatively small mean delays and thus the locations of the intersection points for the different distributions were very similar. See Table~\ref{tab:intpts}. These intersection points are points of {\em double Hopf} bifurcation with the following
implications. If $W^E$ is increased with $W^{IE}_L<W^{IE}<W^{IE}_R$ an asynchronous Hopf curve is crossed {\em before} the synchronous Hopf curve. This will lead to the creation of an asynchronous limit cycle near the Hopf curve, which will be asymptotically stable if the bifurcation is supercritical. Further, there should be curves of Neimark-Sacker (torus) bifurcation emanating from the double Hopf points \cite{guckenheimer2013nonlinear,kuznetsov1998elements}. The torus will combine properties of the synchronous and asynchronous Hopf. For the case of no delay with $N=10$, we studied this with Matcont \cite{dhooge2003matcont}. Matcont indicated the coefficients of the normal form near the two Hopf-Hopf points correspond to the case where a stable torus exists between two curves of Neimark-Sacker bifurcations. We were able to continue one of the curves, but not the other.  See Appendix \ref{appendix_HopfHopf} for more details. 

The effect of the delay on the Hopf curves means that increasing the delay makes intersections points of the asynchronous and synchronous Hopf curves more likely in the uni-directional ring, for example, intersections can occur for smaller sizes of rings. For the bi-directional ring this is not the case. The delay does not seem to make intersection points more or less likely.

\subsection{Numerical Simulations}
To verify the predictions of the numerical bifurcation analysis, we carried out numerical integration of the full system of delay differential equations~\eqref{eq1}, using the command \textsf{NDSolve} in \textit{Wolfram Mathematica}, for $3500$ (uniformly distributed) random points  
in the instability region in the $(W^{IE},W^E)$-plane. For each parameter set, 
 after the numerical solution stabilized, we determined
 the  
 time $\hat{t}$ where the maximum of $E_1(t)$ occurs and  calculated the value
\[a=\max\{|E_1(\hat{t})-E_k(\hat{t})|:2\le k\le n\}.\]
We  
then classified the solution as synchronized when $a<0.001$ and desynchronized otherwise. 

In the bidirectional ring case we only observe synchronized solutions throughout the parameter space, with no delay and with mean delay $\tau_m=0.1$.  See Figure~\ref{fig_hopf_Curves_bi_dir_ODE}A-C.  We only show results for the uniform distribution, as the results for the other distributions are the same -- only synchronous solutions occur. As can be seen in Figure~\ref{fig_hopf_Curves_bi_dir_ODE}D, for this value of the mean delay the Hopf bifurcation curves for all distributions are very similar, thus it is not surprising that the results are similar.
The fact that only synchronous behaviour occurs is consistent with our analysis. Since the asynchronous Hopf curves lie above the synchronous Hopf curve, the equilibrium point loses stability via a synchronous Hopf bifurcation and the asynchronous bifurcation occurs when the equilibrium point is already unstable.
\begin{figure}
    \centering
   \includegraphics[width=1\textwidth]{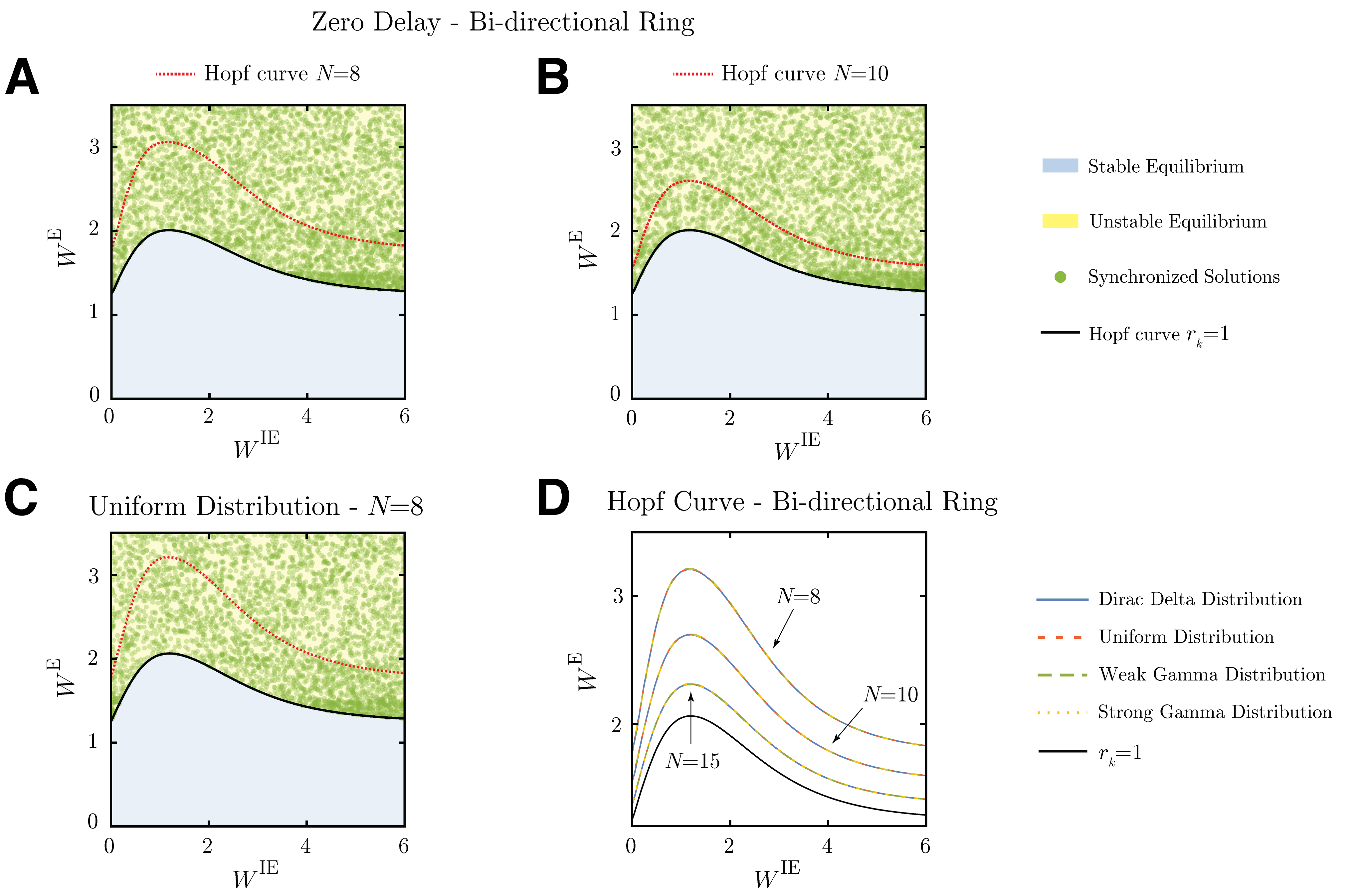}
    \caption{\textbf{Simulation results for the bi-directional ring}.
         $W^{IE}-W^E$ parameter plane, showing the stability region, bifurcation curves and presence of synchronized solutions of the system \eqref{eq1} for \textbf{(A)}, \textbf{(B)} $N=8,10$ neurons with no delay; \textbf{(C)} 
          $N-8$ neurons with the uniform distribution and $\tau_m=0.1$. The black (red) line represents Hopf curve when $r_k=1$ ($r_k=\cos(\nicefrac{2\pi}{N})$). \textbf{(D)} Hopf bifurcation curves with $\tau_m=0.1$ for all distributions considered and numbers of neurons as shown.}
       \label{fig_hopf_Curves_bi_dir_ODE}   
\end{figure}

The uni-directional ring case is more interesting.
The results for the case of zero delay are shown in Figure~\ref{fig_hopf_Curves_uni_dir_ODE} and those for mean delay $\tau_m=0.1$ are shown in Figure~\ref{fig_hopf_Curves_uni_dir}.
In Figure~\ref{fig_hopf_Curves_uni_dir_ODE}A ($N=8$ neurons) we find only synchronized solutions close to the Hopf curves. This agrees with the prediction of our analysis since the equilibrium loses stability via a synchronous Hopf bifurcation.  By contrast, in  Figure~\ref{fig_hopf_Curves_uni_dir_ODE}B ($N=10$ neurons)  we see a region of desynchronized solutions close to the Hopf curves, in the region where the equilibrium point loses stability via the asynchronous Hopf (between the intersection points of the Hopf curves). As can be seen in 
Figure~\ref{fig_hopf_Curves_uni_dir_ODE}C, the asynchronous solutions include both travelling wave type periodic orbits (Point 3) and torus-like solutions (Points 2 and 4) which appear to combine asynchronous and synchronous properties. These simulations are consistent with our study of the Hopf-Hopf interaction points in Appendix~\ref{appendix_HopfHopf}. Interestingly, in both the uni-directional cases, we found a region farther from the Hopf curves where asynchronous solutions were observed.   Numerical bifurcation analysis indicates that these solutions are associated with a torus bifurcation not connected to the Hopf-Hopf points (blue curve in Figure~\ref{fig_hopf_Curves_uni_dir_ODE}B). No such region was found in the bidirectional case, as shown in Figure~\ref{fig_hopf_Curves_bi_dir_ODE}.

To study the effect of the delay, we consider the case of a uni-directional ring with $N=8$ neurons. As shown in Figure~\ref{fig_hopf_Curves_uni_dir}A, the asynchronous Hopf curves with $N=8$ and $\tau_m=0.1$ intersect the synchronous Hopf curves, in contrast with the case of no delay, Figure~\ref{fig_hopf_Curves_uni_dir_ODE}A. As the Hopf curves for all the distributions are very similar, we only show simulation results for the uniform distribution case. The results for the other distributions were similar.  In Figure~\ref{fig_hopf_Curves_uni_dir}B,  we see a region of  desynchronized solutions close to the Hopf curves, in the region where the equilibrium point loses stability via the asynchronous Hopf and another region of desynchronization farther from the Hopf curves. Some example solutions are shown in Figure~\ref{fig_hopf_Curves_uni_dir}C. In general, the results for $N=8$ and $\tau_m=0.1$ are very similar to those with $N=10$ and no delay.

\begin{figure}[H]
     \centering
   \includegraphics[width=1\textwidth]{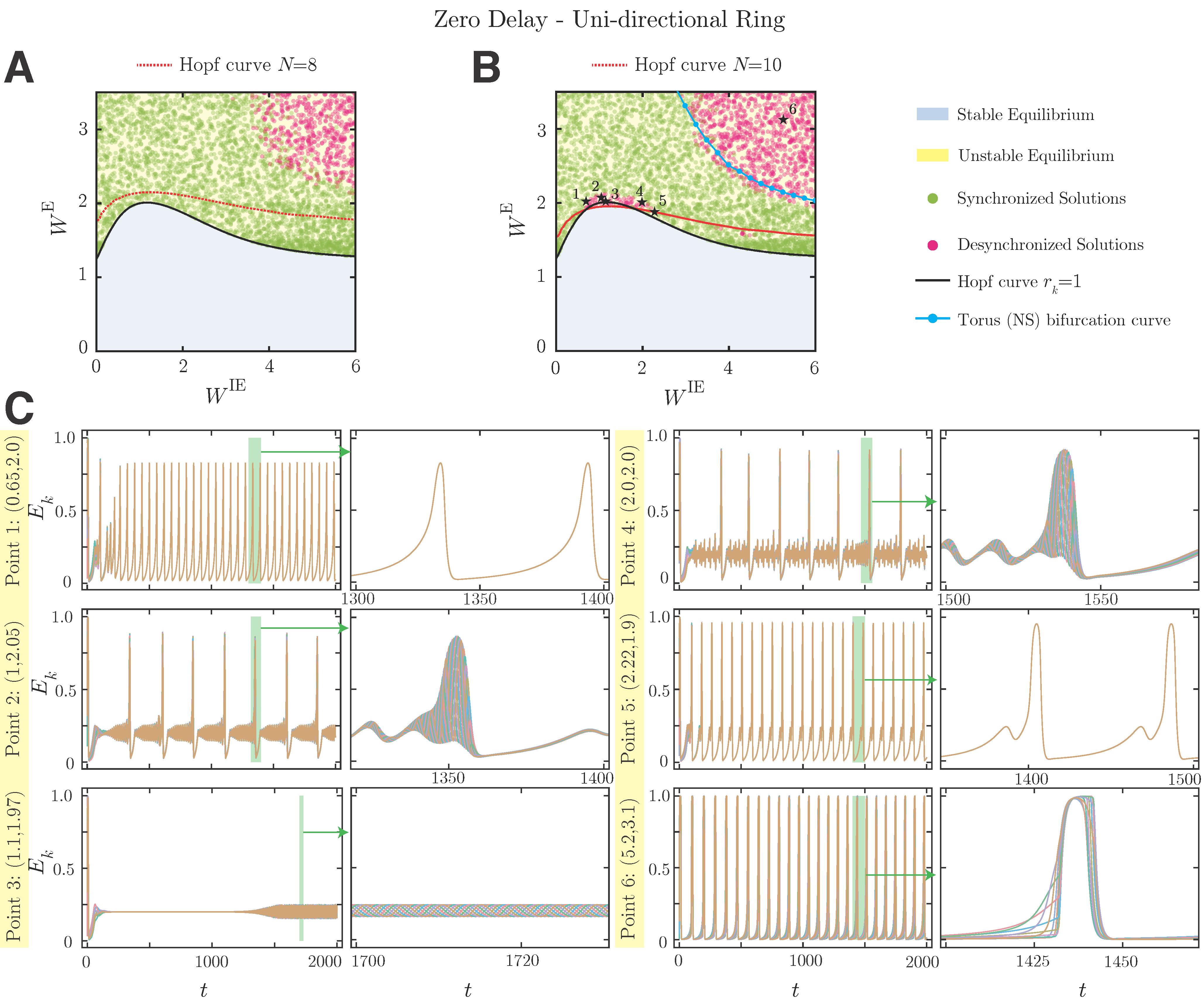}
         \caption{\textbf{Simulation results for the uni-directional ring with no delay}.
         \textbf{(A)}-\textbf{(B)} $W^{IE}-W^E$ parameter plane, showing the stability region, Hopf bifurcation curves and presence of synchronized and desynchronized solutions for $N=8,10$. The black (red) line represents Hopf curve when $r_k=1$ ($r_k=\cos(\nicefrac{2\pi}{N})-i\sin(\nicefrac{2\pi}{N})$). 
          The blue line represents the torus (NS) bifurcation curve. 
         \textbf{(C)} Time series for the solutions corresponding to the parameter sets marked with black stars in B. 
         }
         \label{fig_hopf_Curves_uni_dir_ODE}
\end{figure}

\begin{figure}[H]
     \centering
   \includegraphics[width=1\textwidth]{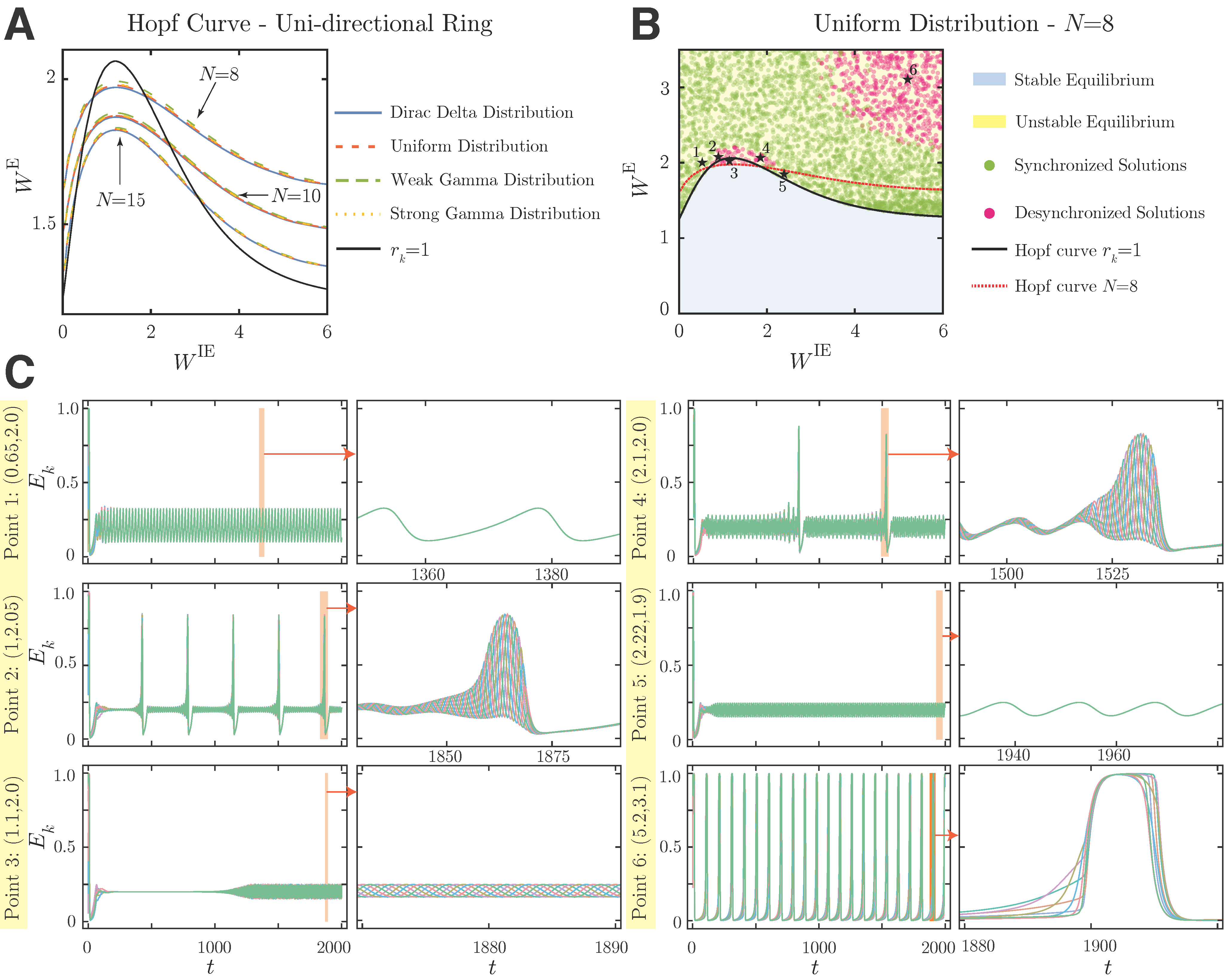}
         \caption{\textbf{Simulations results for the uni-directional ring when $\gavg=0.1$}.
         \textbf{(A)} Hopf curves with different distributions for $N=8,10,15$.
         \textbf{(B)} $W^{IE}-W^E$ parameter plane, showing the stability region, Hopf bifurcation curves and presence of synchronized and desynchronized solutions for $N=8$.
         \textbf{(C)} Time series for the solutions corresponding to the parameter sets marked with black stars in B. 
         }
         \label{fig_hopf_Curves_uni_dir}
\end{figure}

\section{Discussion}

    We considered network of Wilson-Cowan nodes with homeostatic adjustment of the within-node inhibitory to excitatory connection and time delayed connections between excitatory populations of different nodes. We focussed on connectivity matrix with a constant row sum, which ensures that the network always admits a synchronous solution, that is a solution where corresponding variables in each node exhibit the same solution. We used linearization to study the stability and bifurcations of the synchronous equilibrium point. We showed that if the connectivity matrix is diagonalizable,
    the characteristic equations can be written as a product of factors, with each factor corresponding to a different eigenvalue of the connectivity matrix. One factor corresponds to stability within synchronous subspace bifurcations of synchronous solution. Other factors correspond to stability normal to synchronous subspace and bifurcations give rise to asynchronous solutions. We reiterate that this is true for any network where the connectivity matrix is diagonalizable, not just the ring networks that we focussed on for our analysis. For our system, the only bifurcations of the equilibrium point are Hopf bifurcations.
   
    We then contrasted case of a symmetric connectivity matrix, where all eigenvalues are real, with nonsymmetric. In the symmetric case, Hopf bifurcation curves giving rise to asynchronous solutions are situated so that bifurcations always occur after bifurcation of synchronous Hopf, and thus give rise to asymptotically stable solutions. In the asymmetric case, bifurcations curves corresponding to complex eigenvalues may intersect curves of synchronous Hopf giving rise to asynchronous periodic solutions and torus solutions involving asynchronous and synchronous behaviour.
    We showed that this interaction is facilitated by increasing the number of nodes or increasing the time delay. This gives an explanation and possible mechanism for the results found using the Master Stability function in our prior work \cite{nicola2021,al2021impact}.
    We considered a variety of delay distributions: discrete, uniform and gamma.  For small mean delay, all the distributions gave similar results. A full exploration of the differences when the mean delay is not small is left for future work.
    
    We also observed desynchronized solutions not associated with the Hopf-Hopf interaction point. Numerical bifurcation analysis indicates these are torus solutions associated with stand-alone Neimark-Sacker bifurcations of the synchronous periodic solution. In principle these can be studied by linearizing around the synchronous periodic solution. The associated periodic system will have exactly the same structure as the linearization around the equilibrium point. Thus one may expect Neimark-Sacker bifurcations of the synchronous periodic solution given rise to torus solutions where the second frequency is associated with an asynchronous pattern of oscillation. Indeed all our numerical simulations of
    torus solutions showed behaviour associated with the ``splay state", which corresponds to the eigenvalue with largest real part in the connectivity matrix. The study of desynchronization via the linearization about the periodic solution has been explored for networks where the periodic solution of the notes is available in closed form \cite{choe2010controlling}.
    
    Although our analysis focussed on a particular system associated with a brain network model, our approach could be applied any network of oscillators which admits a synchronous equilibrium, such as those studied using the Master Stability function approach \cite{pecora1990synchronization}. Thus the Hopf-Hopf interaction mechanism for desynchronization we identified could occur in these systems as well. The Master Stability function approach essentially linearizes a network about a general synchronous solution and then uses numerical methods to determine how the stability of solution depends on the eigenvalues of the connection matrix. If the synchronous solution is an equilibrium point, then the linearized system can be studied via the characteristic equation as we have done.

\section*{Acknowledgements}
SAC acknowledges the support of the Natural Sciences and Engineering Research Council of Canada. 
IAD acknowledges the support of the Jordan University of Science and Technology through the Faculty Member Research Grant [grant number 20230324]. 
BM acknowledges the support of the Niels Henrik Abel Board.
\bibliography{bib}

\begin{thebibliography}{10}

\bibitem{al2020}
{\sc I.~Al-Darabsah and S.~A. Campbell}, {\em A phase model with large time
  delayed coupling}, Physica D: Nonlinear Phenomena, 411 (2020), p.~132559.

\bibitem{al2021impact}
{\sc I.~Al-Darabsah, L.~Chen, W.~Nicola, and S.~A. Campbell}, {\em The impact
  of small time delays on the onset of oscillations and synchrony in brain
  networks}, Frontiers in Systems Neuroscience,  (2021), p.~58.

\bibitem{anderson1995}
{\sc J.~A. Anderson}, {\em An introduction to neural networks}, MIT press,
  1995.

\bibitem{atay2003distributed}
{\sc F.~M. Atay}, {\em Distributed delays facilitate amplitude death of coupled
  oscillators}, Physical review letters, 91 (2003), p.~094101.

\bibitem{atay2006neural}
{\sc F.~M. Atay and A.~Hutt}, {\em Neural fields with distributed transmission
  speeds and long-range feedback delays}, SIAM Journal on Applied Dynamical
  Systems, 5 (2006), pp.~670--698.

\bibitem{chavez2022complete}
{\sc M.~Ch{\'a}vez-Pichardo, M.~A. Mart{\'\i}nez-Cruz, A.~Trejo-Mart{\'\i}nez,
  D.~Mart{\'\i}nez-Carbajal, and T.~Arenas-Resendiz}, {\em A complete review of
  the general quartic equation with real coefficients and multiple roots},
  Mathematics, 10 (2022), p.~2377.

\bibitem{choe2010controlling}
{\sc C.-U. Choe, T.~Dahms, P.~H{\"o}vel, and E.~Sch{\"o}ll}, {\em Controlling
  synchrony by delay coupling in networks: From in-phase to splay and cluster
  states}, Physical Review E, 81 (2010), p.~025205.

\bibitem{coombes2009delays}
{\sc S.~Coombes and C.~Laing}, {\em Delays in activity-based neural networks},
  Philosophical Transactions of the Royal Society A: Mathematical, Physical and
  Engineering Sciences, 367 (2009), pp.~1117--1129.

\bibitem{destexhe2009wilson}
{\sc A.~Destexhe and T.~J. Sejnowski}, {\em The wilson--cowan model, 36 years
  later}, Biological cybernetics, 101 (2009), pp.~1--2.

\bibitem{dhooge2003matcont}
{\sc A.~Dhooge, W.~Govaerts, and Y.~A. Kuznetsov}, {\em Matcont: a matlab
  package for numerical bifurcation analysis of odes}, ACM Transactions on
  Mathematical Software (TOMS), 29 (2003), pp.~141--164.

\bibitem{ElsNor}
{\sc L.~El'sgol'ts and S.~Norkin}, {\em Introduction to the Theory and
  Application of Differential Equations with Deviating Arguments}, Academic
  Press, New York, 1973.

\bibitem{ermentrout1979}
{\sc G.~B. Ermentrout and J.~D. Cowan}, {\em A mathematical theory of visual
  hallucination patterns}, Biological cybernetics, 34 (1979), pp.~137--150.

\bibitem{GSS88}
{\sc M.~Golubitsky, I.~Stewart, and D.~G. Schaeffer}, {\em Singularities and
  Groups in Bifurcation Theory: Volume II}, Springer-Verlag, New York, 1988.

\bibitem{gotz2009}
{\sc J.~G{\"o}tz, N.~Schonrock, B.~Vissel, and L.~M. Ittner}, {\em Alzheimer's
  disease selective vulnerability and modeling in transgenic mice}, Journal of
  Alzheimer's Disease, 18 (2009), pp.~243--251.

\bibitem{guckenheimer2013nonlinear}
{\sc J.~Guckenheimer and P.~Holmes}, {\em Nonlinear oscillations, dynamical
  systems, and bifurcations of vector fields}, vol.~42, Springer Science \&
  Business Media, 2013.

\bibitem{GH03}
{\sc S.~Guo and L.~Huang}, {\em Hopf bifurcating periodic orbits in a ring of
  neurons with delays}, Physica D, 183 (2003), pp.~19--44.

\bibitem{GH07}
\leavevmode\vrule height 2pt depth -1.6pt width 23pt, {\em Stability of
  nonlinear waves in a ring of neurons with delays}, Journal of Differential
  Equations, 236 (2007), pp.~343--374.

\bibitem{hellyer2016local}
{\sc P.~J. Hellyer, B.~Jachs, C.~Clopath, and R.~Leech}, {\em Local inhibitory
  plasticity tunes macroscopic brain dynamics and allows the emergence of
  functional brain networks}, NeuroImage, 124 (2016), pp.~85--95.

\bibitem{hlinka2012using}
{\sc J.~Hlinka and S.~Coombes}, {\em Using computational models to relate
  structural and functional brain connectivity}, European Journal of
  Neuroscience, 36 (2012), pp.~2137--2145.

\bibitem{kim2014}
{\sc C.~Kim and D.~Q. Nykamp}, {\em Dynamics of a network of excitatory and
  inhibitory neurons induced by depolarization block}, BMC Neuroscience, 15
  (2014), pp.~1--2.

\bibitem{kuznetsov1998elements}
{\sc Y.~A. Kuznetsov}, {\em Elements of Applied Bifurcation Theory}, Springer,
  1998.

\bibitem{kyrychko2013amplitude}
{\sc Y.~Kyrychko, K.~Blyuss, and E.~Sch{\"o}ll}, {\em Amplitude and phase
  dynamics in oscillators with distributed-delay coupling}, Philosophical
  Transactions of the Royal Society A: Mathematical, Physical and Engineering
  Sciences, 371 (2013), p.~20120466.

\bibitem{kyrychko2014synchronization}
\leavevmode\vrule height 2pt depth -1.6pt width 23pt, {\em Synchronization of
  networks of oscillators with distributed delay coupling}, Chaos: An
  Interdisciplinary Journal of Nonlinear Science, 24 (2014).

\bibitem{liang2009phase}
{\sc X.~Liang, M.~Tang, M.~Dhamala, and Z.~Liu}, {\em Phase synchronization of
  inhibitory bursting neurons induced by distributed time delays in chemical
  coupling}, Physical Review E, 80 (2009), p.~066202.

\bibitem{meijer2015modeling}
{\sc H.~G. Meijer, T.~L. Eissa, B.~Kiewiet, J.~F. Neuman, C.~A. Schevon, R.~G.
  Emerson, R.~R. Goodman, G.~M. McKhann, C.~J. Marcuccilli, A.~K. Tryba,
  et~al.}, {\em Modeling focal epileptic activity in the wilson--cowan model
  with depolarization block}, The Journal of Mathematical Neuroscience (JMN), 5
  (2015), pp.~1--17.

\bibitem{meyer2008distributed}
{\sc U.~Meyer, J.~Shao, S.~Chakrabarty, S.~F. Brandt, H.~Luksch, and
  R.~Wessel}, {\em Distributed delays stabilize neural feedback systems},
  Biological Cybernetics, 99 (2008), pp.~79--87.

\bibitem{nicola2021}
{\sc W.~Nicola and S.~A. Campbell}, {\em Normalized connectomes show increased
  synchronizability with age through their second largest eigenvalue}, SIAM
  Journal on Applied Dynamical Systems, 20 (2021), pp.~1158--1176.

\bibitem{nicola2018chaos}
{\sc W.~Nicola, P.~J. Hellyer, S.~A. Campbell, and C.~Clopath}, {\em Chaos in
  homeostatically regulated neural systems}, Chaos: An Interdisciplinary
  Journal of Nonlinear Science, 28 (2018), p.~083104.

\bibitem{pearson2016sensory}
{\sc J.~Pearson, R.~Chiou, S.~Rogers, M.~Wicken, S.~Heitmann, and
  B.~Ermentrout}, {\em Sensory dynamics of visual hallucinations in the normal
  population}, Elife, 5 (2016), p.~e17072.

\bibitem{pecora1990synchronization}
{\sc L.~M. Pecora and T.~L. Carroll}, {\em Synchronization in chaotic systems},
  Physical review letters, 64 (1990), p.~821.

\bibitem{pyragas1995control}
{\sc K.~Pyragas}, {\em Control of chaos via extended delay feedback}, Physics
  letters A, 206 (1995), pp.~323--330.

\bibitem{rahman2017aging}
{\sc B.~Rahman, K.~Blyuss, and Y.~Kyrychko}, {\em Aging transition in systems
  of oscillators with global distributed-delay coupling}, Physical Review E, 96
  (2017), p.~032203.

\bibitem{ryu2020}
{\sc H.~Ryu and S.~A. Campbell}, {\em Effect of time delay on the
  synchronization of excitatory-inhibitory neural networks}, Mathematical
  Biosciences and Engineering, 17 (2020), pp.~7931--7957.

\bibitem{smith2011introduction}
{\sc H.~Smith}, {\em An introduction to delay differential equations with
  applications to the life sciences}, springer, 2011.

\bibitem{vogels2011inhibitory}
{\sc T.~P. Vogels, H.~Sprekeler, F.~Zenke, C.~Clopath, and W.~Gerstner}, {\em
  Inhibitory plasticity balances excitation and inhibition in sensory pathways
  and memory networks}, Science, 334 (2011), pp.~1569--1573.

\bibitem{wang2017}
{\sc Z.~Wang and S.~A. Campbell}, {\em Symmetry, {H}opf bifurcation, and the
  emergence of cluster solutions in time delayed neural networks}, Chaos: An
  Interdisciplinary Journal of Nonlinear Science, 27 (2017), p.~114316.

\bibitem{wille2014synchronization}
{\sc C.~Wille, J.~Lehnert, and E.~Sch{\"o}ll}, {\em
  Synchronization-desynchronization transitions in complex networks: an
  interplay of distributed time delay and inhibitory nodes}, Physical Review E,
  90 (2014), p.~032908.

\bibitem{wilson1972}
{\sc H.~R. Wilson and J.~D. Cowan}, {\em Excitatory and inhibitory interactions
  in localized populations of model neurons}, Biophysical journal, 12 (1972),
  pp.~1--24.

\bibitem{wilson2021evolution}
\leavevmode\vrule height 2pt depth -1.6pt width 23pt, {\em Evolution of the
  wilson--cowan equations}, Biological cybernetics, 115 (2021), pp.~643--653.

\bibitem{Wu98}
{\sc J.~Wu}, {\em Symmetric functional-differential equations and neural
  networks with memory}, Trans. Amer. Math. Soc., 350 (1998), pp.~4799--4838.

\end{thebibliography}
\bibliographystyle{siam}

\appendix
\numberwithin{equation}{section}
\numberwithin{figure}{section}
\numberwithin{table}{section}
\makeatletter
\def\@seccntformat#1{\@ifundefined{#1@cntformat}%
   {\csname the#1\endcsname\quad} 
   {\csname #1@cntformat\endcsname} 
}
\newcommand{\section@cntformat}{Appendix \thesection:\ } 
\makeatother

\section{Stability results for system with zero delay} \label{appendix_stabnodelay}
The case of zero delay corresponds to the delay distribution being given by the Dirac delta function $g(s)=\delta(s)$. In this case the $k^{th}$ factor of the characteristic equation \eqref{charac_k} is 
a cubic polynomial in $\lambda$ given by \eqref{charac_eq_ODE}, which we restate here:
\[    \lambda^3+p_2\lambda^2+p_1\lambda+p_0-r_k W^Eq\lambda(\lambda+1)=0. \]

We have the following result. 
\begin{lemma}\label{lem:zerort}
Let $r_k=\alpha+i\beta$. If  
\begin{equation}\label{condition_NoDelay2}
   \alpha W^EK_1 <1.
\end{equation}
then \eqref{charac_eq_ODE} has no roots with zero real part.
\end{lemma}
\begin{proof}
Recall from \eqref{eq:pdef} that $p_0>0$. Thus zero roots of \eqref{charac_eq_ODE} are not possible. Pure imaginary roots $\lambda=i\omega$ exist only if the following equations are satisfied
\begin{eqnarray}
(p_2-\alpha W^E q)\omega^2-\beta W^E q\omega -p_0 &=& 0\label{eq:H01}\\
\omega^2-\beta W^{E}q\omega+\alpha W^{E}q- p_1&=& 0.\label{eq:H02}
\end{eqnarray}

Since $p_0>0$, if $\beta=0$, then these equations have a solution for $\omega^2$ only if 
$p_2-\alpha W^E q>0$, $p_1-\alpha W^{E}q>0$ 
and $(p_2-\alpha W^E q)(p_1-\alpha W^E q)-p_0=0$.
Using the expressions for $p_2$ and $p_1$ from \eqref{eq:pdef} we have
\begin{align}
\label{eq_coef2}
&\left(p_2-\alpha W^E q\right)
\left(p_1-\alpha W^E q\right)-p_0\nonumber\\
&\hspace{1cm}
=\cfrac{1}{\tau_1^2}\bigg(1+\tau_1-\alpha W^EK_1\bigg)\bigg(1-\alpha W^EK_1 \bigg)
+\cfrac{{W}^{EI*}K_1K_2}{\tau_1^2}\bigg(1+\tau_1-\alpha W^EK_1 \bigg)
+p_0\bigg(1-\alpha W^EK_1 \bigg)
\end{align}
Since ${W}^{EI*},K_1,\tau_1$ are all positive, and $K_2\ge 0$  \eqref{eq_coef2} is positive when
\eqref{condition_NoDelay2}  holds. Thus no pure imaginary roots exist.

If $\beta\ne 0$, solving \eqref{eq:H02} for $\beta W^E q\omega$ and substituting into \eqref{eq:H01} gives
\[(p_2-1-\alpha W^Eq)\omega^2+p_1-p_0-\alpha W^E q=0. \]
Using \eqref{eq:pdef} then shows $\omega$ must satisfy
\begin{equation}
 \frac{1}{\tau_1}\left(1-\alpha W^E K_1\right)\omega^2 +\frac{1}{\tau_1}\left(1-\alpha W^E K_1+W^{EI*}K_1K_2 \right)=0 
 \label{weq0}
\end{equation}
Since $W^{EI*},K_1,K_2$ are all positive, if \eqref{condition_NoDelay2} is satisfied then \eqref{weq0} has no real solution for $\omega$ and no pure imaginary roots exist.
\end{proof}

\begin{proposition}\label{thm:negrtsfactorC}
If   
\eqref{condition_NoDelay2} is satisfied
then all the roots of \eqref{charac_eq_ODE} have negative real part.
\end{proposition}
\begin{proof}
Suppose $W^E=0$. Then \eqref{charac_eq_ODE} becomes
\[ \lambda^3+p_2\lambda+p_1\lambda+p_0=0.\]
From \eqref{eq:pdef} it is easy to check that $p_2>0,\ p_0>0$ and $p_2p_1-p_0>0$. Thus by the Routh-Hurwitz criterion all the roots of \eqref{charac_eq_ODE} have negative real part if $W^E=0$. Since \eqref{charac_eq_ODE} depends continuously on $W^E$ the only way the number of roots with positive real part can change as $W^E$ increases is if at some value of $W^E>0$ there exist roots with zero real part. By Lemma~\ref{lem:zerort}, however, if \eqref{condition_NoDelay2} is satisfied, there are no roots with zero real part. The result follows.
\end{proof}
The proof of  Theorem~\ref{thm:stabnodelay}
follows from Proposition ~\ref{thm:negrtsfactorC} and the fact that $|r_k|\le 1$ for all $r_k$.

\section{Stability for the system with small mean delay}
\label{appendix_stabsmalldelay}

Here we give the proof Theorem~\ref{thm:stabsmalldelay} of that if the mean delay, $\tau_m$ is sufficiently small then 
the stability condition for the system with zero delay still holds.

\begin{proof}
For the case of a Dirac delta distribution, $g(s)=\delta(s-\tau_m)$, the result 
follows from a general small delay result that can be found in  \cite[Section III.4]{ElsNor} or \cite[Theorem 4.4]{smith2011introduction}. The proof  can be extended to the other cases as we outline below.

Write the $k^{th}$ factor of the characteristic equation \eqref{charac_k} as
\begin{equation}
h_k(\lambda,\tau_m)=\det[\lambda \bI-\bL_0-r_kW^E\hat{G}(\lambda,\tau_m)\bM]=
\lambda^3+p_2\lambda^2+p_1\lambda+p_0- r_k qW^E\lambda(\lambda+1)\hat{G}(\lambda,\tau_m)=0.
\label{chareq_gen}
\end{equation} 
where $\bL_0$, $\bM$ are defined in \eqref{matrixLoM}, $p_j$ are defined \eqref{eq:pdef} and
\begin{equation}
\begin{array}{rcl}
\hat{G}(\lambda,\tau_m)&=& e^{-\lambda\tau_m} \quad \mbox{(Dirac delta)}\\
\hat{G}(\lambda,\tau_m)&=& e^{-\lambda\tau_m}\sinhc(\lambda\sigma)\quad \mbox{(uniform)}\\
\hat{G}(\lambda,\tau_m)&=& \frac{1}{1+\lambda\tau_m} \quad \mbox{(weak gamma)}\\
\hat{G}(\lambda,\tau_m)&=& \frac{1}{(1+\frac12\lambda\tau_m)^2} \quad \mbox{(strong gamma)}
\end{array}\label{Gchoice}
\end{equation}

To begin note from Theorem~\ref{thm:stabnodelay} that if $W^EK_1<1$ then for each $k$ the three roots of 
\[  h_k(\lambda,0)=
\lambda^3+p_2\lambda^2+p_1\lambda+p_0- r_k   W^E q\lambda(\lambda+1)=0\]
all have negative real parts. 

To apply the proof of \cite[Theorem 4.4]{smith2011introduction} to
the characteristic equation \eqref{chareq_gen} requires two results
\begin{enumerate}
\item $h_k(\lambda,\tau_m)$ is continuous
\item Given a sequence of points $(z_n,\tau_n)$ such that $h_k(z_n,\tau_n)=0$, $\Re(z_n)\ge\xi$ for some fixed $\xi$, 
$\tau_n\rightarrow 0$ and $|z_n|\rightarrow\infty$, we have
$\lim_{n\rightarrow\infty}\frac{\hat{G}(z_n,\tau_n)}{z_n}=0$
\end{enumerate}

1. For the uniform distribution, $h_k(\lambda,\tau_m)$ is continuous, since the exponential and sinc functions are continuous.
Further, the characteristic equations for the weak and strong gamma cases are defined for $\Re(\lambda)>-1/\tau_m$, and on this region $h_k(\lambda,\tau_m)$ is continuous. 

2. Let  $z=x_n+iy_n$, with $x_n\ge \xi>-\frac{1}{\tau_n}$. 
Consider the weak gamma distribution. For each $n$ we have
\begin{align*}   
\left|\frac{\hat{G}(z_n,\tau_n)}{z_n}\right| & =  \frac{1}{\left|z_n\right|\left|1+z_n\tau_n\right|}\\
&= \frac{1}{\left|z_n\right|\sqrt{(1+\tau_nx_n)^2+(\tau_ny_n)^2}}\\
& \le \frac{1}{\left|z_n\right||1+\tau_n x_n|}\\
& \le \frac{1}{\left|z_n\right||1+\tau_n \xi|}
\end{align*}
The result follows since $\tau_n\rightarrow 0$ and $|z_n|\rightarrow\infty$ as $n\rightarrow\infty$.
The proof for the strong gamma distribution is similar.

For the uniform distribution, recall that $\sigma<\tau_m$, thus for each $\tau_n$ we must have a corresponding $\sigma_n<\tau_n$ with $\sigma_n\rightarrow 0$ as $n\rightarrow\infty$. Then for each $n$ we have
\begin{align*}   
\left|\frac{\hat{G}(z_n,\tau_n)}{z_n}\right| & =
\left| \frac{1}{2\sigma_nz_n} \int_{\tau_n-\sigma_n}^{\tau_n+\sigma_n}
e^{-z_n s}\, ds\right|\\
&\le  \frac{1}{2\sigma_n|z_n|} \int_{\tau_n-\sigma_n}^{\tau_n+\sigma_n}
|e^{-z_n s}|\, ds\\
&\le \frac{1}{2\sigma_n|z_n|} \int_{\tau_n-\sigma_n}^{\tau_n+\sigma_n}
e^{-x_n s}\, ds\\
&\le \frac{1}{2\sigma_n|z_n|} \int_{\tau_n-\sigma_n}^{\tau_n+\sigma_n}
e^{-\xi s}\, ds\\
&\le \frac{e^{-\xi\tau_n}}{|z_n|}{{\rm sinhc}
\left({\xi\sigma_n}\right)}\\
\end{align*}
The result follows since $\tau_n\rightarrow 0$, $\sigma_n\rightarrow 0$ and $|z_n|\rightarrow\infty$ as $n\rightarrow\infty$.

Thus the proof of \cite[Theorem 4.4]{smith2011introduction} can be applied to characteristic equation \eqref{charac_k} in all cases to show the following. Let $\lambda^0_1,\lambda^0_2,\lambda^0_3$ be the roots of $h_k(\lambda,0)=0$.  Then for small enough mean delay, $\tau_m$, each root $\lambda$ of $h_k(\lambda,\tau_m)=0$ is either close to one of the $\lambda^0_i$
or satisfies
  $\Re(\lambda)<\min_i\Re(\lambda^0_i)$. It follows that if $W^EK_1<1$ all roots of \eqref{charac_k} for each $k$
  have negative real part.
\end{proof}

\section{Curves of Pure Imaginary Eigenvalues.}
\label{appendix_C}
In this section we determine properties of the curves of pure imaginary eigenvalues when there is no delay in the system. Then we
find expressions for the curves of pure imaginary eigenvalues when the system has a various specific distributions of delays. 
\subsection{Zero delay}
As shown in Appendix~\ref{appendix_stabnodelay} the $k^{th}$ factor of the characteristic equation in this case is \eqref{charac_eq_ODE}, which has pure imaginary eigenvalues if equations \eqref{eq:H01}-\eqref{eq:H02} are satisfied.
First we consider some simple consequences of one of the stability results in Appendix~\ref{appendix_stabnodelay}.
\begin{proposition}\label{prop_negalp}
If $\alpha\le 0$, there are no  curves of pure imaginary eigenvalues in the region $W^E>0$.
\end{proposition}
\begin{proof}
If $\alpha\le 0$ and $W^E>0$ the $\alpha W^E K_1<0$ and by Lemma~\ref{lem:zerort}
the characteristic equation has no pure imaginary roots.
\end{proof}

\begin{proposition}\label{Th_order_curves}
The curves $W^{E}_{Hopf}(W^{IE})$ defined by \eqref{eq:WEhopf}-\eqref{eq:omega} satisfy
\[ \lim_{W^{IE}\rightarrow 0} W^{E}_{Hopf}(W^{IE})=\lim_{W^{IE}\rightarrow\infty} W^{E}_{Hopf}(W^{IE})=\frac{1}{\alpha K_1}. \]
\end{proposition}
\begin{proof}
Recall from the proof of Proposition~\ref{thm:endpt} that 
\[ \lim_{W^{IE}\rightarrow 0} p_1
= \lim_{W^{IE}\rightarrow \infty} p_1
=\frac{1}{\tau_1}+p_0,
\]
which is independent of $W^E$. Putting these limits into \eqref{eq:omega} show that it simplifies to
\[ \alpha(\omega^2+1)\left(\omega^2-\frac{\beta}{\alpha\tau_1}\omega-p_0\right)=0 \quad \Rightarrow \omega^2=\frac{\beta}{\alpha\tau_1}\omega+p_0. \]
The following equivalent expression to \eqref{eq:WEhopf} can be found by solving \eqref{eq:H02} for $W^E$:
\[ W^E_{Hopf}=\frac{\omega^2-p_1}{q(\beta\omega-\alpha)}\]
Using the expressions for $p_1$ and $\omega^2$ above then gives
\[ W^{E}_{Hopf}=\frac{1}{\alpha K_1}\frac{\beta\omega+\alpha\tau_1(p_0-p_1)}{\beta\omega-\alpha}=\frac{1}{\alpha K_1}.\]
\end{proof}
\begin{proposition}
The curves of pure imaginary eigenvalues satisfy
\[ \frac{1}{\alpha K_1} \le W^E_{Hopf}(W^{IE}), \]
with equality only at the endpoints.
\end{proposition}
\begin{proof} 
By Proposition~\ref{prop_negalp}, we need only consider $\alpha>0$. In this case, the inequality follows from Lemma~\ref{lem:zerort}. From Proposition~\ref{Th_order_curves}, it is clear that equality holds at the endpoints. To see that it only holds there, note that $K_2>0$ for $W^{IE}>0$ and the proof of Lemma~\ref{lem:zerort}, carries over with $\alpha W^E K_1\le 1$ if $K_2>0$.
\end{proof}
\begin{proposition}
If  $r_k=\alpha\in\mathbb{R}$,  then the curves of pure imaginary eigenvalues satisfy
\[ W^E_{Hopf}(W^{IE}) < \frac{1+\tau_1}{\alpha K_1}. \]
\end{proposition}
\begin{proof} 
Setting $\beta=0$ in \eqref{eq:H01} and rearranging we have
\[ \omega^2=\frac{p_0}{p_2-W^E q\alpha}\]
From \eqref{eq:pdef} $p_2,p_0$ and $q$ are all positive. Thus \eqref{eq:H01} can only be satisfied if $W_E<p_2/(q\alpha) $. Putting this together with Lemma~\ref{lem:zerort} gives the result.
\end{proof}

In the case of a real eigenvalue of the connectivity matrix, $r_k=\alpha\in{\mathbb R}$, we can analyze the curves of pure imaginary eigenvalues in some detail. In this case, equations \eqref{eq:H01}-\eqref{eq:H02} may be rewritten
\begin{align}
p_0-p_2 \Omega&=-\alpha W^E q\,\Omega \label{eq:ReIm_1r} \\
p_{1a}+p_{1b} W^Eq-\Omega&= \alpha W^E q\label{eq:ReIm_2r}.
\end{align}
where $\Omega=\omega^2$. Eliminating $W^{E}$ gives  
\begin{equation} \Omega^2+B_2\Omega+B_0=0 \label{Om_eq} \end{equation}
with
\[ B_2=-p_{1a}+p_2\left(1-\frac{p_{1b}}{\alpha}\right),\qquad
B_0=p_0\left(\frac{p_{1b}}{\alpha}-1\right). \]
Solving \eqref{Om_eq} gives
\begin{equation} 
\Omega_{\pm}= -\frac{B_2}{2}\pm\sqrt{\frac{B_2^2}{4}-B_0}. 
\label{Om_rts}
\end{equation}
Note that only positive $\Omega_{\pm}$ correspond to the characteristic equation having pure imaginary roots. 
It follows that for a given set of parameters there can be up to
two values of $W^E$ at which pure imaginary eigenvalues occur:
\begin{equation} 
W^{E}_{Hopf\pm}=\frac{1}{q\alpha}\left(p_2-\frac{p_0}{\Omega_\pm}\right), 
\label{WEhopf}
\end{equation}
or, equivalently,
\begin{equation} 
W^{E}_{Hopf\pm}=\frac{p_{1a}-\Omega\pm}{q(\alpha-p_{1b})}.
\label{WEhopf2}
\end{equation}
The existence and properties of these points depend on the signs of $B_2$ and $B_0$.

\begin{proposition}\label{prop_onecurve}
If $r_k=\alpha\in{\mathbb R}$ there is at most one curve of pure imaginary eigenvalues of \eqref{charac_eq_ODE} lying in the region $W^E>0$.
\end{proposition}
\begin{proof}
We consider two cases, depending on the size of $\alpha$. Recall that $\alpha>0$ for the curve to lie in $W^E>0$ and $0<p_{1b}<1.$

If $p_{1b}\le \alpha$ then  $B_0\le 0$, which implies
$\Omega_- \le 0<\Omega_+$.  Thus there are no pure imaginary eigenvalues corresponding to $\Omega_-$.

If $0<\alpha<p_{1b}$ then we have $B_0>0$, $B_2<0$ and  
\[ -\frac{B_2}{B_0}=\frac{p_2}{p_0}+\frac{p_{1a}}{\frac{p_{1b}}{\alpha}-1}>\frac{p_2}{p_0}.\]
If $B_0>\frac14 B_2^2$, no roots of \eqref{Om_eq} exist and there are no pure imaginary roots of the characteristic equation. If $B_0= \frac14 B_2^2$, there is only one root of \eqref{Om_eq}.
If $B_0< \frac14 B_2^2$, the solutions of \eqref{Om_eq} satisfy
$0<\Omega_-<-\frac12 B_2<\Omega_+$ and
\[ \Omega_- = -\frac{B_2}{2}+\frac{B_2}{2}\sqrt{1-\frac{4B_0}{B_2^2}}
=-\frac{B_2}{2}+ \frac{B_2}{2}\left(1- 2\frac{B_0}{B_2^2}-2\frac{B_0^2}{B_2^4} -\cdots\right)
= -\frac{B_0}{B_2}-\frac{B_0^2}{B_2^3}-\cdots
<-\frac{B_0}{B_2}< \frac{p_0}{p_2}.
\]
It follows from \eqref{WEhopf} that the curves of pure imaginary eigenvalues corresponding to $\Omega_-$ lie in $W^E<0$.

\end{proof}
\begin{proposition}
If $r_k=\alpha\in{\mathbb R}$ then $W^E_{Hopf}(W^{IE})$ is a decreasing function of $\alpha$
\end{proposition}
\begin{proof}
Assume $B_0\le \frac14 B_2^2$. It follow from Proposition~\ref{prop_onecurve}
that we need only consider $\Omega_+$. 

Now we consider how $W^{E}_{Hopf+}$ varies with $\alpha$ if all the other parameters are fixed. Differentiating \eqref{WEhopf} gives
\[ \frac{dW^{E}_{Hopf+}}{d\alpha}=
-\frac{1}{\alpha}W^{E}_{Hopf+}+\frac{p_0}{q\alpha\Omega^2_+}\frac{d\Omega_+}{d\alpha} \]
Differentiating  \eqref{Om_eq} implicitly with respect to $\alpha$ and using \eqref{WEhopf} yields

\begin{eqnarray*}
\displaystyle\frac{d\Omega_+}{d\alpha}&=&\frac{p_{1b}}{\alpha^2}\,\frac{p_0-p_2\Omega_+}{2\Omega_++B_2}\\
&=&-\frac{p_{1b}}{\alpha}\frac{W^E_{Hopf}\, q\,\Omega_+}{2\Omega_++B_2}
\end{eqnarray*}
Then from \eqref{Om_rts} we have
\[
\displaystyle\frac{d\Omega_{+}}{d\alpha}=-\frac{p_{1b}}{\alpha}\frac{W^E_{Hopf}\, q\,\Omega_{+}}{\sqrt{\frac14 B_2^2-B_0}}
\]
It follows that $\Omega_+$ is a decreasing
function of $\alpha$ which in turn implies $W^{E}_{Hopf}$ is a decreasing function of $\alpha$.
\end{proof}

\subsection{Weak gamma distribution}
\label{appendix_C1}
 The characteristic equation \eqref{gammaeq4} can be written as 
\begin{equation}\label{charac_weak_kernel}
   \Delta_1(\lambda):=\lambda^4+a_3 \lambda^3 +a_2 \lambda^2+ a_1 \lambda+a_0=0
\end{equation}
where
\begingroup
\allowdisplaybreaks
\begin{align}\label{coeff_weak_kernel}
a_3&=1+\gamma+\cfrac{1}{\tau_1},\qquad a_0=\cfrac{K_1\gamma {I^*}^2}{\tau_1\tau_2}\\
    a_2&:=a_{20}+a_{21}W^E:=\Bigg(\gamma+\cfrac{1+\gamma}{\tau_1}+\cfrac{K_1  {I^*}^2}{\tau_1\tau_2}-
\cfrac{K_1K_2 \phi^{-1}(p)}{\tau_1 \phi(W^{IE}p)}\Bigg)+
\Bigg( 
\cfrac{K_1K_2 p}{\tau_1 \phi(W^{IE}p)}-\cfrac{K_1 r_k\gamma}{\tau_1}  \Bigg)W^E\nonumber\\
    a_1&:=a_{10}+a_{11}W^E:=\Bigg( \cfrac{\gamma}{\tau_1}+\cfrac{K_1 (1+\gamma) {I^*}^2}{\tau_1\tau_2}-
\cfrac{K_1K_2 \gamma \phi^{-1}(p)}{\tau_1 \phi(W^{IE}p)} \Bigg)+
\Bigg( 
\cfrac{K_1K_2 p\gamma}{\tau_1 \phi(W^{IE}p)}-\cfrac{K_1 r_k\gamma}{\tau_1}    \Bigg)W^E.\nonumber
\end{align}
\endgroup

Assume $r_k=\alpha\in{\mathbb R}$. Then
substituting $\lambda=i \omega$ ($\omega>0$ and $i=\sqrt{-1}$) into the characteristic polynomial \eqref{charac_weak_kernel}  yields the following:
\begingroup
\allowdisplaybreaks
\begin{subequations}
\begin{align}
   \text{Real part: }&a_0 - a_2 \omega^2 + \omega^4=0\label{wk_real_part}\\
   \text{Imaginary part: }&a_1 \omega - a_3  \omega^3=0\label{wk_img_part}.
\end{align}
\end{subequations}
\endgroup 
Substituting $\omega=\sqrt{a_1/a_3}$, obtained from \eqref{wk_img_part}, in \eqref{wk_real_part} leads to
\begin{equation}\label{eqn_356}
    a_0+\cfrac{a_1^2}{a_3^2}-\cfrac{a_1a_2}{a_3}=0.
\end{equation}
Recall that $a_i$ ($i=1,2$) depends on $W^E$. Using $ a_i=a_{i0}+a_{i1}W^E$ provided in \eqref{coeff_weak_kernel} and solving  \eqref{eqn_356} for $W^E$ the Hopf bifurcation curve in terms of $W^{IE}$:
\begin{neweq}\label{weak_kernel_hopf_curve}
      W^E=W_{\rm Hopf}^{E} \Big(W^{IE}\Big):=& \cfrac{1}
   {2 a_{11} \big(a_{11} - a_{21} a_3\big)} 
   \Bigg( a_{11} a_{20} a_{3} -2 a_{10} a_{11}  + a_{10} a_{21} a_{3}\\
   ~& + a_3 \sqrt{\big( a_{11} a_{20} - a_{10} a_{21} \big)^2 - 4 a_0 a_{11} \big(a_{11} - a_{21} a_{3}\big) }\,\,  \Bigg)
\end{neweq}
Note that
\[ a_3=\gamma+p_2,\quad a_0=\gamma p_0,\quad a_{20}=\gamma+p_{1a}
\quad a_{10}=p_0+\gamma p_{1a}\]
are all positive, while the signs of 
\[ a_{21}=q(p_{1b}-\gamma \alpha),\quad a_{11}=\gamma q(p_{1b}-\alpha)\]
depend on $\gamma$ and $\alpha$.

When $r_k=\alpha+i\beta$ is complex, we write 
\[a_{11}=\gamma \hat{a}_{11}-r_k \bar{a}_{11}\qquad\text{and}\qquad a_{21}= \hat{a}_{11}-r_k \bar{a}_{11}\]
where
\[\hat{a}_{11}=\cfrac{K_1K_2 p}{\tau_1 \phi(W^{IE}p)}\qquad\text{and}\qquad
\bar{a}_{11}=\cfrac{K_1\gamma}{\tau_1}.\]
In this case, the characteristic polynomial \eqref{charac_weak_kernel} at $\lambda=i \omega$ ($\omega>0$ and $i=\sqrt{-1}$) gives
\begingroup
\allowdisplaybreaks
\begin{subequations}
\begin{align}
\text{Real part: }&a_0+  \bar{a}_{11}\beta W^E  \omega +(\bar{a}_{11}\alpha W^E - 
 \hat{a}_{11} W^E- a_{20}) \omega^2 + \omega^4 =0\label{wk_real_part_complex}\\
   \text{Imaginary part: }&w\left(a_{10} +  \hat{a}_{11}\gamma W^E -  \bar{a}_{11}\alpha W^E+  \bar{a}_{11}\beta W^E \omega-a_3 \omega^2   \right)
=0\label{wk_img_part_complex}.
\end{align}
\end{subequations}
\endgroup 
Solving \eqref{wk_img_part_complex} for $\omega$ yields
\[\omega_\pm=\frac{\beta   \bar{a}_{11} W^E \pm \sqrt{\beta ^2 \bar{a}_{11}^2  W^{E^2}
+4\left(  \gamma \hat{a}_{11} a_3-\alpha   \bar{a}_{11} a_3  \right)W^E+4 a_{10} a_3}}{2 a_3}.\]
Substituting $\omega_\pm$ in \eqref{wk_real_part_complex} leads to following equation 
\begin{equation}
    A_4W^{E^4}+A_3 W^{E^3}+ A_2 W^{E^2}+ A_1 W^{E}+ A_0=
    \pm W^{E} \left(B_2 W^{E^2}+ B_1 W^{E}+ B_0\right)\sqrt{C_2 W^{E^2}+ C_1 W^{E}+ C_0}
\end{equation}
where
\begin{neweq_no}
    &A_0=\cfrac{a_{10}^2 - a_{10} a_{20} a_3 + a_0 a_3^2}{a_3^2},\qquad
    A_1= \cfrac{\alpha \bar{a}_{11} a_{10} a_{3}   + 
  \alpha\bar{a}_{11} a_{20} a_{3}-\hat{a}_{11} a_{10} a_{3}  - 2 \alpha \bar{a}_{11} a_{10}   + 2 \gamma \hat{a}_{11} a_{10} -  \gamma \hat{a}_{11} a_{20} a_{3}}{a_3^2}  \\
  & A_2=\cfrac{\bar{a}_{11}^2 (4 \beta^2  a_{10}+ a_{3}^2 (\beta^2-2 \alpha^2 ) + 
    a_{3} (2 \alpha^2 -\beta^2 a_{20} )) + 
 2 \hat{a}_{11}^2 a_{3} \gamma ( \gamma-a_{3}) + 
 2\alpha \bar{a}_{11} \hat{a}_{11}a_{3}  (a_{3} - 2 \gamma +  \gamma a_{3}) }{2a_3^3}\\
  &A_3=\cfrac{  \alpha \beta^2
 \bar{a}_{11}^3 a_3 + 4 \beta^2 \gamma \bar{a}_{11}^2 \hat{a}_{11}-\beta^2\bar{a}_{11}^2 \hat{a}_{11}a_3  - 4\alpha \beta^2 \bar{a}_{11}^3 }{2a_3^3}, 
  \qquad A_4=\cfrac{\bar{a}_{11}^4  \beta^4}{2 a_3^4},
  \\
   &B_0=\cfrac{ \beta \bar{a}_{11}  ( a_{20} a_{3} -2 a_{10}- a_{3}^2 ) }{2a_3^3},  \qquad
   B_1=\cfrac{\bar{a}_{11}  \hat{a}_{11} a_{3} \beta  + 2 \bar{a}_{11}^2 \alpha  \beta - 
 \bar{a}_{11}^2 a_{3} \alpha \beta  - 2 \bar{a}_{11}  \hat{a}_{11} \beta  \gamma  }{2a_3^3},\qquad
  B_2= \cfrac{-\bar{a}_{11}^3  \beta^3}{2a_3^4}  \\
  &C_0=4 a_{10} a_3,\qquad C_1=4\left(  \gamma \hat{a}_{11} a_3-\alpha   \bar{a}_{11} a_3  \right),\qquad C_2=\beta ^2 \bar{a}_{11}^2
    \end{neweq_no}

\subsection{Strong gamma distribution}
\label{appendix_B2}
The characteristic equation \eqref{gammaeq5} can be written as  
\begin{equation}\label{charac_strong_kernel}
   \Delta_2(\lambda):=\lambda^5+b_4\lambda^4+b_3 \lambda^3 +b_2 \lambda^2+b_1 \lambda+b_0=0
\end{equation}
where
\begingroup
\allowdisplaybreaks
\begin{align}\label{coeff_strong_kernel}
b_4&=1+2\gamma+\cfrac{1}{\tau_1},\qquad b_0=\cfrac{K_1\gamma^2 {I^*}^2}{\tau_1\tau_2}\nonumber\\
 b_3&:=b_{30}+b_{31}W^E:=\Bigg(\gamma^2+2\gamma+\cfrac{1+2\gamma}{\tau_1}+\cfrac{K_1  {I^*}^2}{\tau_1\tau_2}-
\cfrac{K_1K_2 \phi^{-1}(p)}{\tau_1 \phi(W^{IE}p)}\Bigg)+\Bigg( \cfrac{p\, K_1 K_2 }{\tau_1 \phi(W^{IE}p)} \Bigg)W^E \nonumber \\
    b_2&:=b_{20}+b_{21}W^E:=\Bigg(\cfrac{2\gamma+ (1+\tau_1)\gamma^2}{\tau_1}+
\cfrac{(1+2\gamma)K_1  {I^*}^2}{\tau_1\tau_2}-
\cfrac{2K_1K_2 \gamma \phi^{-1}(p)}{\tau_1 \phi(W^{IE}p)}\Bigg)\\
&\hspace{8.5cm}+ \Bigg( \cfrac{ 2p\gamma K_1K_2-K_1 r_k \gamma^2 \phi(W^{IE}p)   }{\tau_1 \phi(W^{IE}p)}   \Bigg)W^E \nonumber\\
    b_1&:=b_{10}+b_{11}W^E:= \Bigg(\cfrac{\gamma^2\tau_2+\gamma (2+\gamma) {I^*}^2 K_1 }{\tau_1\tau_2}-\cfrac{K_1K_2 \gamma^2\phi^{-1}(p)}{\tau_1 \phi(W^{IE}p)}\Bigg)   \nonumber\\
 &\hspace{8.5cm}+\Bigg( \cfrac{ K_1\gamma^2  \tau_2 \big(K_2 p - r_k \phi(W^{IE}p) \big)}{\tau_1 \tau_2\phi(W^{IE}p) }   \Bigg)W^E.\nonumber
\end{align}
\endgroup
Substituting $\lambda=i \sqrt{\omega}$ ($\omega>0$ and $i=\sqrt{-1}$) into the characteristic polynomial \eqref{charac_strong_kernel}  yields the following:
\begingroup
\allowdisplaybreaks
\begin{subequations}
\begin{align}
   \text{Real part: }&b_0 - b_2 \omega  + b_4 \omega^2=0\label{sk_real_part}\\
   \text{Imaginary part: }& \sqrt{\omega} \big(b_1  - b_3  \omega+\omega^2\big)=0\label{sk_img_part}.
\end{align}
\end{subequations}
\endgroup 
From \eqref{sk_img_part}, we have 
\begin{equation}\label{eq:0098}
    \omega=\cfrac{b_3+\sqrt{b_3^2-4b_1}}{2}.
\end{equation}
 Substituting $\omega$ defined in \eqref{eq:0098}   into equation \eqref{sk_real_part} gives 
\begin{equation}\label{eq_1198}
  b_0-\cfrac{b_2}{2}\left(b_3+\sqrt{b_3^2-4b_1} \right) +\cfrac{b_4}{4}\left(b_3+\sqrt{b_3^2-4b_1} \right)^2=0.
\end{equation}
Notice that equation \eqref{eq_1198} can be written as
\begin{equation}\label{eq_2298}
  4b_0-2b_2b_3+2b_4b_3^2-4b_1b_4=2\big( b_2-b_3b_4 \big)\sqrt{b_3^2-4b_1}
\end{equation}
Recall that $b_i$ ($i=1,2,3$) depends on $W^E$. Using $b_i=b_{i0}+b_{i1}W^E$ provided in \eqref{coeff_strong_kernel} and squaring both sides of equation \eqref{eq_2298} lead to
\begin{equation}\label{eq_3398}
   B_3W_E^3+ B_2W_E^2+B_1W_E+B_0=0
\end{equation}
where
\begingroup
\allowdisplaybreaks
\begin{align*}
B_3=&16 b_{11} b_{21} \Big(b_{21} - b_{31} b_{4}\Big)\\
B_2=&16 \Big(b_{11}^2 b_4^2 + \big(b_{10} b_{21} - b_0 b_{31}\big) \big(b_{21} - b_{31} b_4\big) - 
   b_{11} \big(-2 b_{20} b_{21} + b_{21} b_{30} b_4 + b_{20} b_{31} b_4\big)\Big)\\
B_1=&16 \Big(  b_{10} \big(2 b_{20} b_{21} - b_{21} b_{30} b_4 \textcolor{cyan}{-} b_{20} b_{31} b_4\big)  
+b_0 \big( 2 b_{30} b_{31} b_4 - b_{21} b_{30} - b_{20} b_{31} \big)\\
~&\hspace{6.8cm} +b_{11} \big(b_{20}^2 - b_{20} b_{30} b_4 - 2 b_4 \big[b_0 - b_{10} b_4\big] \big)\Big) \\
B_0=&16 \Big(b_0^2 - b_0 b_{20} b_{30} + b_0 b_4 \big( b_{30}^2-2 b_{10} \big)  + 
   b_{10} \big(b_{20}^2 - b_{20} b_{30} b_4 + b_{10} b_4^2\big)\Big)
\end{align*}
\endgroup
Since \eqref{eq_3398} always has one real solution, one Hopf bifurcation curve  is given by:
\begin{neweq}\label{strong_kernel_hopf_curve}
      W^E=W_{\rm Hopf}^{E} \Big(W^{IE}\Big):=&  \cfrac{\left(12B_3\sqrt{3B_4} -B_5\right)^{\frac{1}{3}}} {6B_3}
      - \cfrac{2\big(3B_3B_1-B_2^2\big)}{3B_3\left(12B_3\sqrt{3B_4} -B_5\right)^{\frac{1}{3}}}-\cfrac{B_2}{3B_3}
\end{neweq}
where
\begin{align*}
    B_4&=27B_3^2B_0^2 - 18 B_3 B_2 B_1 B_0 + 4 B_3 B_1^3 + 4 B_2^3 B_0 - B_2^2 B_1^2\\
    B_5&=108 B_3^2 B_0  + 36 B_3 B_2 B_1 - 8 B_2^3.
\end{align*}
There can be up to two other curves.

When $r_k=\alpha+i\beta$ is complex, we write 
\[b_{11}=\gamma \hat{b}_{11}-r_k \bar{b}_{11}\qquad\text{and}\qquad b_{21}= \hat{b}_{11}-r_k \bar{b}_{11}\]

\[\hat{b}_{11}=\cfrac{2p\gamma K_1K_2 }{\tau_1 \phi(W^{IE}p)}\qquad\text{and}\qquad
\bar{b}_{11}=\cfrac{\gamma^2 K_1}{\tau_1}.\]

Substituting $\lambda=i\omega$ ($\omega>0$ and $i=\sqrt{-1}$) into the characteristic polynomial \eqref{charac_strong_kernel}  yields the following:
\begingroup
\allowdisplaybreaks
\begin{subequations}
\begin{align}
   \text{Real part: }&b_0 +\beta \bar{b}_{11} W^E\omega+(\alpha \bar{b}_{11} W^E - \hat{b}_{11} W^E-b_{20})\omega^2+ b_4 \omega^4=0\label{sck_real_part}\\
   \text{Imaginary part: }& \omega \big(b_{10}+\gamma \hat{b}_{11} W^E-\alpha \bar{b}_{11}W^E+\beta\bar{b}_{11}W^E \omega-(b_{30}+b_{31}W^E)\omega^2 +\omega^4\big)=0\label{sck_img_part}.
\end{align}
\end{subequations}
\endgroup 
Following \cite{chavez2022complete} and using \textit{Wolfram Mathematica}, we solve \eqref{sck_img_part} for $\omega$ and get 
\[\omega_\pm( W^E):=-\frac{h_4}{2}\pm \frac{\sqrt{h_5-h_6}}{2}
\qquad\text{and}\qquad
\omega_\pm( W^E):=\frac{h_4}{2}\pm \frac{\sqrt{h_5+h_6}}{2}\]
where
\begin{neweq_no}
&~h_6=-\frac{2\beta \bar{b}_{11} W^E}{h_4},\qquad
h_5=\frac{4}{3}\left(b_{30}+b_{31}W^E\right)-h_3,\qquad
h_4=\sqrt{h_3+\frac{4}{3}\left(b_{30}+b_{31}W^E\right)}\\
&~ h_3=\frac{1}{3}\sqrt[3]{\frac{h_2}{2}}+\frac{1}{3}\sqrt[3]{\frac{2}{h_2}}
\left[ \left(b_{30}+b_{31}W^E\right)^2+12\left(  b_{10}-\alpha \bar{b}_{11}W^E+\gamma  \hat{b}_{11}  W^E \right)  \right]\\
&~ h_2=h_1+\sqrt{h_1^3-4\left[ \left(b_{30}+b_{31}W^E\right)^2+12\left(  b_{10}-\alpha \bar{b}_{11}W^E+\gamma  \hat{b}_{11}  W^E \right)  \right]^3}\\
&~ h_1=27\beta^2 \bar{b}_{11}^2 W^{E^2} - 72  \left(b_{30}+b_{31}W^E\right)
-12 \left(b_{30}+b_{31}W^E\right) \left(\alpha \bar{b}_{11} W^E-b_{10}-\gamma   \hat{b}_{11}  W^E \right)
\end{neweq_no}
Substituting $\omega_\pm( W^E)$ in \eqref{sck_real_part} leads to following implicit equation of $W^E$:
\begin{equation}
  b_0 +\beta \bar{b}_{11} W^E\omega_\pm( W^E)+(\alpha \bar{b}_{11} W^E - \hat{b}_{11} W^E-b_{20})\omega_\pm( W^E)^2+ b_4 \omega_\pm( W^E)^4=0
\end{equation}

\section{Accuracy of The Numerical Approach}
\label{appendix_NumericalAccuracy}
To verify the accuracy of the numerical approach described in subsection~\ref{sec:numbif} for finding the curves of pure imaginary eigenvalues,  we used Maple to solve \eqref{eq:WEhopf}-\eqref{eq:omega} with the parameter values in \eqref{params} for the cases where we could find explicit expressions for these curves. Recall that we have $r_k=1$, $r_k=\cos(\frac{2\pi}{N})$ in the bi-directional ring and $r_k=\cos(\frac{2\pi}{N})+i\sin(\frac{2\pi}{N})$ in the unidirectional ring. When there is no delay in the system,  the curves are found by solving \eqref{eq:omeganodelay} for the real value(s) of $\omega$ and then substituting into \eqref{eq:WEhopfnodelay}. For the weak gamma distribution in the bi-directional case the curve is given explicitly by \eqref{weak_kernel_hopf_curve}. 
The results are shown in Figure~\ref{Fig_comparison_plots}.
For the parameter values we considered, in the unidirectional case we found two values for $\omega$ and hence two curves, shown by the solid red and orange curves in  Figure~\ref{Fig_comparison_plots}. In the bidirectional case and for $r_k=1$ there was only one curve as predicted by the analysis.
\begin{figure}[H]
     \centering
   \includegraphics[width=1\textwidth]{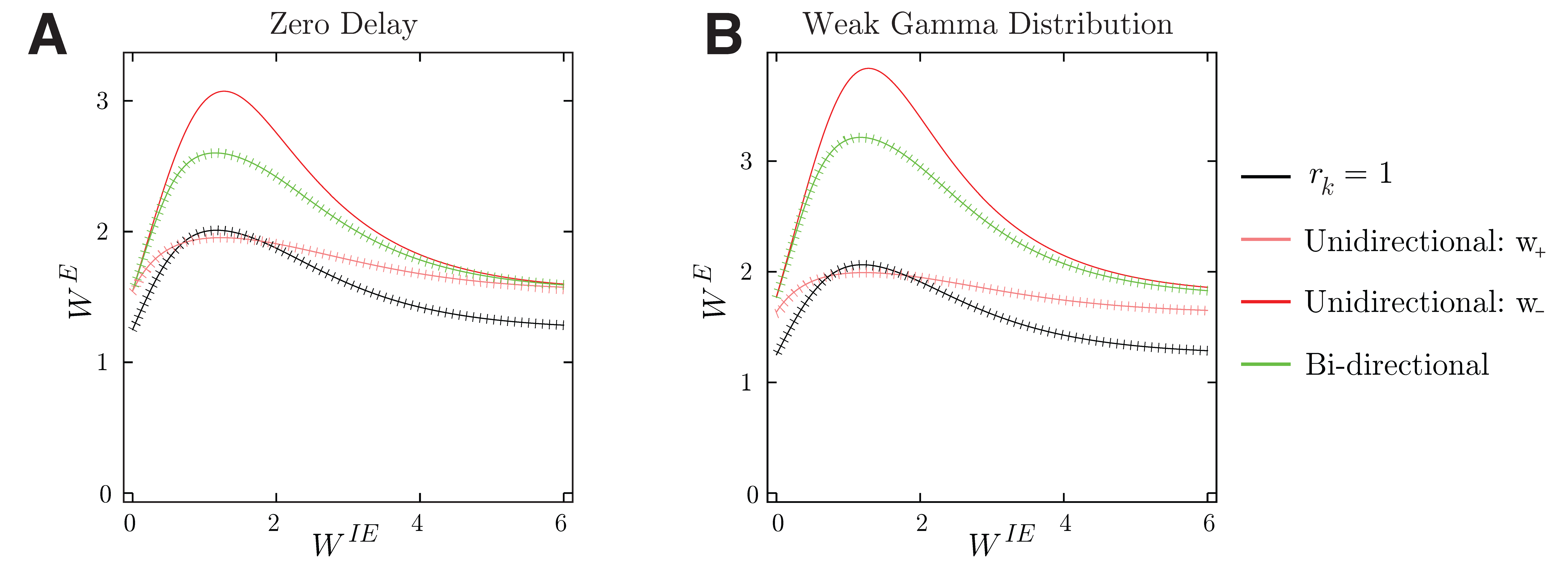}
      \caption{Solid curves are exact solutions of equations for the curves of pure imaginary eigenvalues and dashed curves are numerical results. \textbf{(A)} $N=10$, no delay; \textbf{(B)} $N=8$, weak gamma distributed delay.} 
      \label{Fig_comparison_plots}
\end{figure}

\eject
\section{Hopf-Hopf Bifurcation Analysis of the System with Zero Delay}\label{appendix_HopfHopf}
We used Matcont\cite{dhooge2003matcont} to perform numerical bifurcation analysis of system \eqref{eq1} with $N=10$ and zero delay. The results are shown in Figure~\ref{Fig_Hopf_coefficient}
\begin{figure}[H]
     \centering
   \includegraphics[width=1\textwidth]{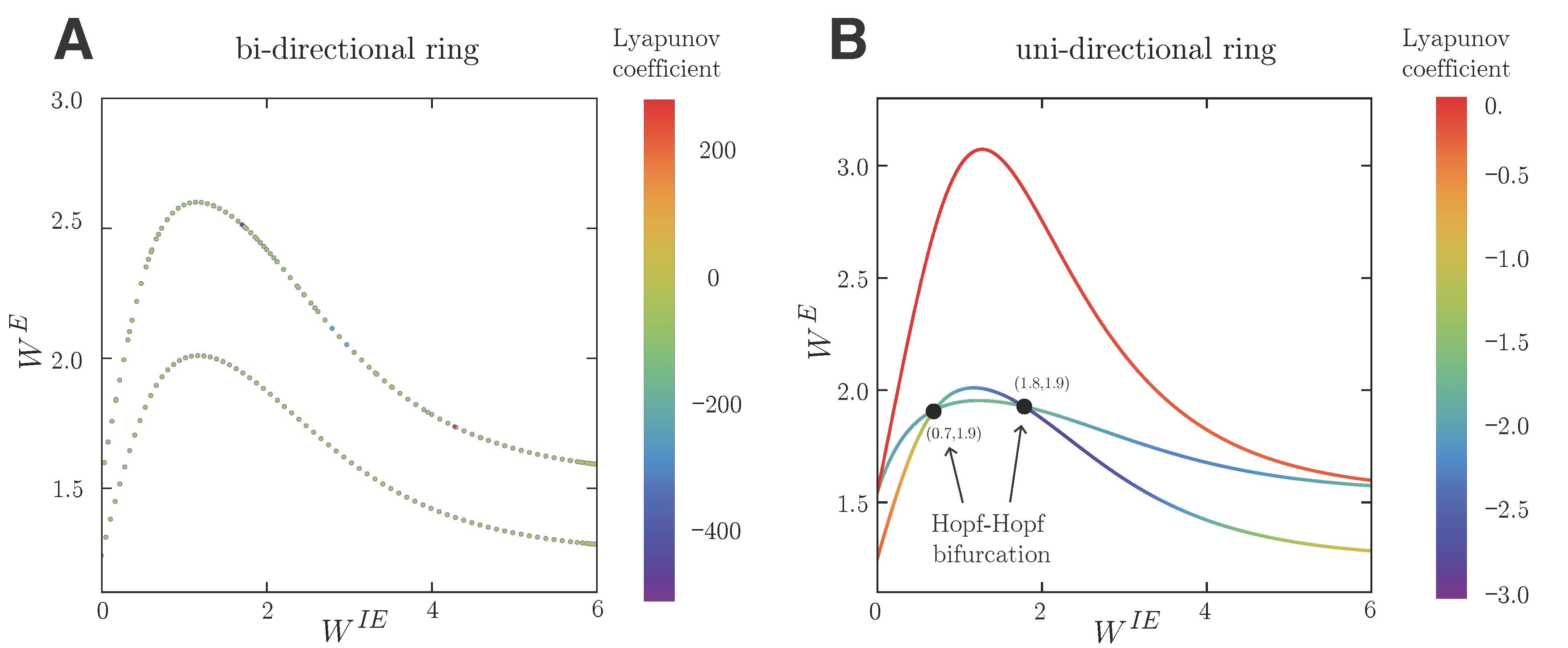}
      \caption{The sign of the Hopf coefficient along the Hopf branches in the zero delay case.}
      \label{Fig_Hopf_coefficient}
\end{figure}

Matcont gives the coefficients for the Hopf-Hopf point in the rescaled amplitude equations of the normal form near a Hopf-Hopf point \cite[Equation (8.111)]{kuznetsov1998elements}:  
\begin{align}
   \xi_1'&= \xi_1\big(\mu_1-\xi_1-\theta\, \xi_2+\Theta\, \xi_2^2      \big),\label{HopfHopfbifurcation_1}\\
   \xi_2'&=\xi_2\big(\mu_2-\delta\,\xi_1- \xi_2+\Delta\, \xi_1^2      \big).\label{HopfHopfbifurcation_2}
\end{align}
For our model with $N=10$ at the points indicated in Figure~\ref{Fig_Hopf_coefficient} the values are given in Table~\ref{tab:HH}.
\begin{table}[H]
\centering
\begin{tabular}{|c|c|c|c|c|c|}
\hline
Hopf-Hopf point  & ${\rm sign}(\nicefrac{p_{11}}{p_{22}})$  &  $\theta$  &   $\delta$  &   $\Theta$  &   $\Delta$  \\ \hline
    $(0.692,1.911)$   & positive & $0.895$ & $1.338$ & $-47.0596$ & $-156.606$ \\ \hline
  $(1.780,1.928)$  & positive & $0.4217$ &  $2.8168$ &  $-12.1523$ & $-302.628$ \\ \hline
\end{tabular}
\caption{Coefficients of the normal form \eqref{HopfHopfbifurcation_1}-\eqref{HopfHopfbifurcation_2} for the Hopf-Hopf points.}
\label{tab:HH}
\end{table}
Since both Hopf bifurcations are supercritical everywhere, as shown in Figure~\ref{Fig_Hopf_coefficient}, $p_{11},p_{22}<0.$
Thus we have $0<\theta<\delta$ and $\theta\delta>1$.
This corresponds to case II in \cite[Figures 8.25 \& 8.26]{kuznetsov1998elements}, which we illustrate in Figure~\ref{fig:HHschem}. In particular, there is a stable torus existing between two curves of Neimark-Sacker (torus) bifurcation.
\begin{figure}[tb]
\centering
\includegraphics{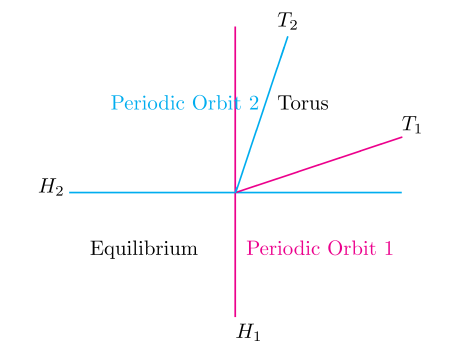}
    \caption{Schematic of bifurcations and stable solutions near Hopf-Hopf point. $H_j$: Hopf bifurcation leading to stable periodic orbit $j$. $T_j$: Torus (Neimark-Sacker) bifurcation of periodic orbit $j$ leading to stable torus.}
    \label{fig:HHschem}
\end{figure}

\end{document}